\documentclass[12pt]{article} 
\usepackage{cmap}
\usepackage[top=2.5cm,bottom=2.5cm,left=2.5cm,right=2.5cm]{geometry}
\usepackage{fix-cm}
\usepackage[T1]{fontenc}
\usepackage{microtype}
\usepackage[utf8]{inputenc}
\usepackage[english]{babel}
\usepackage{amsfonts,amsmath,amsthm,amssymb}
\usepackage{mathtools}

\usepackage{tikz}
\usetikzlibrary{positioning, intersections, arrows.meta, decorations.markings, shapes.geometric, arrows}

\usepackage{float}
\usepackage{enumitem}
\usepackage{adjustbox}
\usepackage[labelfont=bf, labelsep=period]{caption}

\usepackage{xcolor}
\usepackage{authblk}

\usepackage[pdftex,colorlinks,linkcolor=blue,citecolor=red,urlcolor=blue,hyperindex,plainpages=false,bookmarksopen,bookmarksnumbered,unicode]{hyperref}

\usepackage{booktabs}
\usepackage{multirow}
\usepackage{multicol}

\definecolor{processblue}{cmyk}{0.96,0,0,0}
\definecolor{codegreen}{rgb}{0,0.6,0}
\definecolor{codegray}{rgb}{0.5,0.5,0.5}
\definecolor{codepurple}{rgb}{0.58,0,0.82}
\definecolor{backcolour}{rgb}{0.95,0.95,0.92}
\definecolor{codegray}{rgb}{0.5,0.5,0.5}
\definecolor{keywordcolor}{rgb}{0.1,0.1,0.8}

\newcommand{\ceil}[1]{\left\lceil{#1}\right\rceil}
\newcommand{\floor}[1]{\left\lfloor{#1}\right\rfloor}
\newcommand{\ran}[2]{\{#1, \ldots, #2\}}

\newtheorem{theorem}{Theorem}[section]
\newtheorem{lemma}[theorem]{Lemma}
\newtheorem{proposition}[theorem]{Proposition}
\newtheorem{corollary}[theorem]{Corollary}

\theoremstyle{definition}

\theoremstyle{remark}

\begin{document}
	\large
	\title{\vspace{-2cm} There is no $8$-regular $K_3$-irregular graph}
	
	\author[1, 2]{Artem Hak \thanks{Corresponding author: artikgak@ukr.net}}
	\author[1, 2]{Sergiy Kozerenko}
	\author[2]{Andrii Serdiuk}
	
	\affil[1]{\footnotesize National University of Kyiv-Mohyla Academy, Skovorody str. 2, 04070 Kyiv, Ukraine.}
	
	\affil[2]{\footnotesize Kyiv School of Economics, Mykoly Shpaka str. 3, 03113 Kyiv, Ukraine.}
	
	\date{\vspace{-5ex}}
	
	\maketitle

    \begin{abstract}
    A graph is $K_3$-irregular if its vertices belong to pairwise distinct numbers of triangles. We prove that no $8$-regular $K_3$-irregular graph exists, settling the last unresolved case. Following the initial discovery of such graphs for regularities $r \in \{10,11,12\}$ (Stevanovi\'c et al., 2024), our previous work (Hak et al., 2025) showed that no such graphs exist for $r \le 7$, provided the first example for $r=9$, and proved that any $8$-regular candidate must have between $17$ and $22$ vertices. We exclude these possible orders for $r=8$ by combining careful analysis of triangle degrees with integer linear programming techniques. Meanwhile, a recent construction (Zhang, 2026) established that regular $K_3$-irregular graphs do exist for all $r \ge 9$. Together with our results, this establishes that an $r$-regular $K_3$-irregular graph exists if and only if $r\geq 9$.
    \end{abstract}

    
	
	\section{Introduction}

    A graph is regular if all its vertices have the same degree. At the other extreme, a simple graph with more than one vertex cannot have pairwise distinct vertex degrees. A different way to measure irregularity in a graph is to count copies of another fixed graph $F$ containing each vertex~\cite{Char:87}. The $F$-degree of a vertex $v$ is the number of subgraphs isomorphic to $F$ that contain $v$, and a graph is $F$-irregular if its vertices have pairwise distinct $F$-degrees.

    The ordinary vertex degree is the $K_2$-degree. The next case among complete graphs is the $K_3$-degree, which counts the triangles containing a vertex. In their 1988 paper, Chartrand, Erd\H{o}s, and Oellermann~\cite{Char-Erd-Oell:88} asked whether a graph can be both regular and $K_3$-irregular. Such graphs are also called regular triangle-distinct graphs: every vertex has the same number of neighbors, but belongs to a different number of triangles.

    The existence question remained open for decades. Early on, Nair and Vijayakumar studied triangle degrees in a graph and its complement~\cite{Nair:94}, as well as triangle counts associated with edges~\cite{Nair:96}. The problem was later highlighted in Chartrand's retrospective survey on irregular graphs~\cite{Chartrand2016}, and more recently, Berikkyzy et al.~\cite{dist-triangle:2024} studied triangle-distinct graphs, determined the smallest such graph, and independently raised the question of whether regular examples exist. 

    A related line of work studies $(r_2,r_3)$-constant graphs, in which every vertex has the same ordinary degree $r_2$ and belongs to the same number $r_3$ of triangles~\cite{HakReg:25,CaroMifsud2025,Caro2024}. When $r_3>0$, these are precisely vertex-girth-regular graphs of girth $3$, a class studied in the broader setting by Jajcay et al.~\cite{Jajcay2025}. The question considered here instead asks whether the triangle counts can be pairwise distinct while the ordinary degrees remain equal.

    The existence question was answered affirmatively in 2024 by Stevanovi\'c et al.~\cite{reg-triangle:24}, who found the first regular $K_3$-irregular graphs for $r\in\{10,11,12\}$. Their examples include a $10$-regular graph on $21$ vertices. Subsequently, Hak, Kozerenko, and Serdiuk~\cite{Hak:25} proved that no such graph exists for $r\leq 7$ and constructed a $9$-regular example on $24$ vertices. Using an evolutionary algorithm, they also found examples for every regularity $9\leq r\leq 30$.
    Recently, Zhang~\cite{Zhang:26} established the existence of regular $K_3$-irregular graphs for every $r\geq 9$, extending the previously known range through an infinite construction. 

	In our previous work~\cite{Hak:25}, we showed that an $8$-regular $K_3$-irregular graph, if one exists, must have order $17\leq n\leq 22$. The existence of such a graph remained open. Here we exclude all six possible orders by combining new structural bounds with an integer linear programming (ILP) formulation and exhaustive infeasibility computations. While the smaller orders ($17 \le n \le 19$) fall to initial eliminations, the larger orders ($20 \le n \le 22$) require our newly derived structural bounds on common neighborhoods to become computationally tractable. Consequently, we establish the following main result of this paper.

    \begin{theorem}\label{thm-main}
		There does not exist an $8$-regular $K_3$-irregular graph.
	\end{theorem}

	Combining Theorem~\ref{thm-main} with the known non-existence for $r \le 7$ \cite{Hak:25} and the existence constructions for $r \ge 9$ \cite{Hak:25, Zhang:26}, we obtain the exact threshold for the existence of regular triangle-distinct graphs.

    \begin{theorem}\label{thm-threshold}
		An $r$-regular $K_3$-irregular graph exists if and only if $r\geq 9$.
	\end{theorem}

    The paper is organized as follows. Section~\ref{sec:preliminaries} introduces the necessary definitions and proves new structural properties, including neighborhood intersection bounds that narrow the search space. Section~\ref{sec:ilp-formulation} details the translation of these properties into an ILP formulation. Section~\ref{sec:search-results} presents the computational search strategy, including the base elimination technique and the two-vertex refinements, along with the results establishing the nonexistence of $8$-regular $K_3$-irregular graphs. Finally, Section~\ref{sec:conclusion} provides concluding remarks and open problems.
	
	\section{Preliminaries and Structural Properties}\label{sec:preliminaries}

    In this section, we provide the core graph-theoretic definitions and notation used throughout the paper. We review key previous results that frame our methodology, and derive structural properties and bounding techniques used to heavily constrain the search space in our ILP models.

    \subsection{Definitions}

    A graph is an ordered pair $G = (V, E)$ where $V = V(G)$ denotes the set of its \textit{vertices} and $E = E(G) \subseteq \binom{V}{2}$ is the set of its \textit{edges}. All graphs considered in this paper are finite, simple, and have at least two vertices. For brevity, an edge $\{u, v\}$ is denoted $uv$.

    Two graphs $G$ and $H$ are \textit{isomorphic} if there is a bijection $f: V(G) \rightarrow V(H)$ such that $uv \in E(G)$ $\iff$ $f(u)f(v) \in E(H)$.
    
    The \textit{neighborhood}, or \textit{open neighborhood}, of a vertex $v$ is the set of all its adjacent vertices: $N_G(v) = \{u \in V(G): uv \in E(G)\}$. The \textit{closed neighborhood} is the set $N_G[v] = N_G(v) \cup \{v\}$. When the graph is clear from the context, we omit the subscript $G$. Two distinct vertices $u, v$ are called \textit{true twins} if $N_G[u] = N_G[v]$, and \textit{false twins} if $N_G(u) = N_G(v)$. A \textit{complete graph} $K_n$ is a graph on $n$ vertices in which all vertices are pairwise adjacent.

    For any subset of vertices $S \subseteq V(G)$, we denote by $G[S]$ the subgraph induced by $S$, and we let $E(S) = E(G[S])$ denote the set of edges with both endpoints in $S$. For two disjoint subsets $S_1, S_2 \subseteq V(G)$, we denote by $E(S_1, S_2)$ the set of edges with one endpoint in $S_1$ and the other in $S_2$.
    
    For a vertex $v \in V(G)$, the \textit{degree} of $v$ is the cardinality of its neighborhood: $\deg_G(v) = |N_G(v)|$. A graph is called \textit{$r$-regular} if all its vertices have degree $r$. 
    
    The \textit{$K_3$-degree}, or \textit{triangle degree}, of a vertex $v$, denoted by $K_3\deg(v)$, is the number of subgraphs $H$, isomorphic to $K_3$, to which $v$ belongs. That is, $K_3\deg(v) = |E(N(v))|$. A graph is called \textit{$K_3$-irregular} if all its vertices have distinct $K_3$-degrees.

    In a $K_3$-irregular graph, we denote by $v_d$ the unique vertex of $K_3$-degree $d$, whenever such a vertex exists. Note that in the ILP formulation (Section~\ref{sec:ilp-formulation}), we re-index the vertices as $w_0, w_1, \ldots, w_{n-1}$ according to their structural role in ILP models (not by their $K_3$-degrees). In the computational results (Section~\ref{sec:search-results}), we revert to the $v_d$ notation, where the subscript denotes the $K_3$-degree.

    \subsection{Previous results}
	
	We recall several results from our previous work~\cite{Hak:25} that will be used in the ILP formulation and the subsequent eliminations.
	
	\begin{corollary}[Hak et al. \cite{Hak:25}]~\label{triangle-shake-lemma}
		For any graph, the sum of $K_3$-degrees of all of its vertices is divisible by $3$.
	\end{corollary}
	
	\begin{lemma}[Hak et al. \cite{Hak:25}]~\label{lemma-no-twins}
		If $G$ is regular $K_3$-irregular graph then $G$ has no true twins and no false twins.
	\end{lemma}
	
	\begin{proposition}[Hak et al. \cite{Hak:25}]\label{8-reg-upper-bound}
		Let $G$ be an $8$-regular $K_3$-irregular graph. Then $K_3\deg(v)\leq 22$ for every $v\in V(G)$.
	\end{proposition}
	
	\subsection{The partitioning technique and basic properties}\label{sec:part-tech}
	
	Throughout the proofs and ILP formulation, we rely on the vertex partitioning technique introduced in~\cite{Hak:25}. For a given $r$-regular $K_3$-irregular graph $G$ of order $n$, we can fix an arbitrary vertex $v \in V(G)$, and denote its triangle-degree by $d = K_3 \deg(v)$ and partition the vertex set $V(G)$ into three disjoint subsets: the singleton $\{v\}$, its open neighborhood $A=N(v)$, and the non-neighborhood $B = V(G) \setminus N[v]$.
	
	This naturally decomposes the edge set $E(G)$ into four sets:
	\[E(G) = E(\{v\},A) \sqcup E(A) \sqcup E(A,B) \sqcup E(B).\]
	
	By definition, the induced subgraph $G[A]$ contains exactly $|A|=r$ vertices and $|E(A)| = d$ edges. 
	Since the graph $G$ is $r$-regular, we have 
	\[\sum_{a \in A} \deg(a) = |A| \cdot r = r^2.\] 
	Next, we can calculate the number of edges between $A$ and $B$ by subtracting the edges incident to $v$ and those within $A$:
	\begin{equation*}
		|E(A, B)| = r^2 - r - 2d = r(r-1) - 2d.
	\end{equation*}
	
	Consequently, the size of $B$ and the number of edges within $B$ are given by:
	\begin{align*}
		&|B| = n - r - 1,\\
		&|E(B)| =\frac{nr}{2} - r - d - (r(r-1) - 2d) = \frac{nr}{2} - r^2 + d.
	\end{align*}

    By Lemma~\ref{lemma-no-twins}, neither $\overline{G[A]}$ nor $G[B]$ has an isolated vertex. Indeed, an isolated vertex in $\overline{G[A]}$ would be a true twin of $v$, whereas an isolated vertex in $G[B]$ would be a false twin of $v$.
    
	See Figure~\ref{fig-2} for the visualization of the partitioning technique.
	
	\begin{figure}[ht]
		\begin{tikzpicture}[auto,on grid, state/.style ={circle, top color =black , bottom color = black , draw, minimum width=1mm}, inner sep=1.5pt, scale = 1]
			
			\node[state, label={[xshift=-5pt]left:$v$}] (1) at (0, 0) {};
			
			\node[below left = 0.9cm and 0.35cm of 1] {\small $d = K_3 \deg(v)$};

			\draw[line width=0.5mm] (4, 0) ellipse (2cm and 2.5cm);
			\node[anchor=center] at (4, 0.5) {\small $|A| = r$};
			\node[anchor=center] at (4, -0.5) {\small $|E(A)| = d$};
			
			\node[below=1.25cm of 1] at (4, 4.5) {\large $A$};
			
			\draw[line width=0.5mm] (13, 0) ellipse (2cm and 2.5cm);
			\node[anchor=center] at (13, 0.5) {\small $|B| = n - r - 1$};
			\node[anchor=center] at (13, -0.5) {\small $|E(B)| = \frac{nr}{2} - r^2 + d$};
			
			\node[below=1.25cm of 1] at (13, 4.5) {\large $B$};
			
			\draw[line width=0.5mm] (6, -0.5) -- (11, -0.5);
			\node[] at (8.5, 0) {\large $\ldots$};
			\draw[line width=0.5mm] (6, 0.5) -- (11, 0.5);
			
			\node[] at (8.5, -1.5) {\small $|E(A,B)| = r(r - 1) - 2d$};
			
			\coordinate (T1) at (2.76, 1.95);
			\coordinate (T2) at (2.76, -1.95);
			\draw[line width=0.5mm] (1) -- (T1);
			\draw[line width=0.5mm] (1) -- (T2);
		\end{tikzpicture}
		\caption{The partitioning technique for $r$-regular $K_3$-irregular graphs.}\label{fig-2}
	\end{figure}
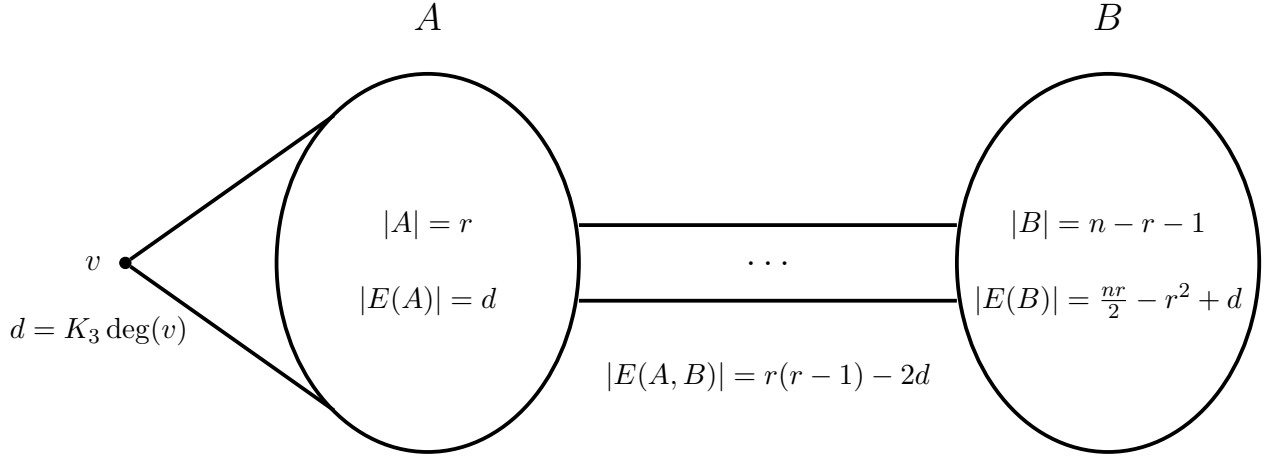

    \subsection{Degree differences and neighborhood intersection bounds}\label{sec:advanced-props}

    We first bound the difference between the $K_3$-degrees of adjacent vertices in an arbitrary regular graph and then strengthen the bound for $8$-regular $K_3$-irregular graphs. We subsequently derive bounds on the common neighborhoods of non-adjacent vertices.

    \begin{lemma}\label{general-edge-bound}
		Let $G$ be an $r$-regular graph and let $uv \in E(G)$. Then 
        \[|K_3 \deg(u) - K_3 \deg(v)| \le \binom{r-1}{2}.\]
	\end{lemma}
    \begin{proof}
		Without loss of generality assume $K_3 \deg(v) \ge K_3 \deg(u)$. Let $S = N(u) \cap N(v)$ with $c = |S|$, and define $A_u = N(u) \setminus (S \cup \{v\})$ and $A_v = N(v) \setminus (S \cup \{u\})$. Since $A_u \subseteq N(u) \setminus N(v)$ and $A_v \subseteq N(v) \setminus N(u)$, the sets $A_u$ and $A_v$ are disjoint, each of cardinality $r - 1 - c$.
        
		Partition $N(v) = \{u\} \sqcup S \sqcup A_v$. Every vertex of $S$ is adjacent to $u$, contributing $c$ edges, while no vertex of $A_v$ is adjacent to $u$.
        Hence
		\[
		K_3 \deg(v) = |E(S)| + |E(S, A_v)| + |E(A_v)| + c.
		\]
		Similarly, decomposing the neighborhood of $u$ gives 
        \[K_3 \deg(u) = |E(S)| + |E(S, A_u)| + |E(A_u)| + c \ge |E(S)| + c.\] 
        Therefore,
        \begin{align*}
            K_3 \deg(v) - K_3 \deg(u) \le |E(S, A_v)| + |E(A_v)| &\le c(r-1-c) + \binom{r-1-c}{2} 
            \\ = \frac{(r-1)(r-2) + c(1-c)}{2} &\le \binom{r-1}{2},
        \end{align*}
		where the last inequality holds because $c(1-c) \le 0$ for every $c \ge 0$.
	\end{proof}
    
    For adjacent vertices in an $8$-regular graph, Lemma~\ref{general-edge-bound} gives $|K_3\deg(u)-K_3\deg(v)|\leq \binom{7}{2} = 21$. For $8$-regular $K_3$-irregular graphs, the next lemma improves this bound to $19$.

\begin{lemma}\label{no-edge-lemma}
	Let $G$ be an $8$-regular $K_3$-irregular graph. If $u, v \in V(G)$ are adjacent, then $|K_3 \deg(u) - K_3 \deg(v)| \le 19$. Moreover, if $K_3 \deg(v) - K_3 \deg(u) = 19$, then $N(v)$ contains a vertex of $K_3$-degree at least $21$.
\end{lemma}
\begin{proof}
	Suppose, for the sake of contradiction, that there exist adjacent vertices $v$, $u$ with $K_3 \deg(v) - K_3 \deg(u) \ge 20$.
	
	Let $c = |N(u) \cap N(v)|$ be the number of common neighbors of $u$ and $v$, and put $U = N(v) \setminus \{u\}$, so $|U| = 7$. Partitioning $N(v) = \{u\} \sqcup U$, the edges induced by $N(v)$ decompose into the $c$ edges between $u$ and $U$ and the $|E(U)|$ edges within $U$. Thus $K_3\deg(v) = |E(U)| + c$. 
    
    Since each of the $c$ common neighbors forms a triangle with the edge $uv$, vertex $u$ participates in at least $c$ triangles, which implies $K_3\deg(u) \ge c$. Consequently, 
    \[|E(U)| \ge K_3\deg(v) - K_3\deg(u) \ge 20.\] 
    As $\binom{7}{2} = 21$, the induced subgraph $G[U]$ has at most one missing edge compared to $K_7$, which implies that at most two vertices of $U$ have degree less than $6$ in $G[U]$. Thus, the set of $6$-degree vertices 
    \[
    D = \{w \in U : \deg_U(w) = 6\}
    \]
    satisfies $|D| \ge 7 - 2 = 5$. 

	Every vertex $w \in D$ is adjacent to $v$ and to all other $6$ vertices in $U$, which accounts for $7$ of its $8$ neighbors. Therefore, each $w \in D$ has a unique remaining neighbor $z_w \notin U \cup \{v\}$. Note that $z_w \ne u$, since $z_w = u$ would imply $N[w] = N[v]$, contradicting Lemma~\ref{lemma-no-twins}. Hence, $z_w \in V(G) \setminus (U \cup \{u, v\})$.

    \begin{figure}[ht]
\centering
\begin{tikzpicture}[
    scale=1.1,
    vertex/.style={circle, fill=black, inner sep=1.8pt, draw=black},
    node-v/.style={circle, fill=black, inner sep=2.5pt, draw=black},
    node-w/.style={circle, fill=blue!85!black, inner sep=2.5pt, draw=black},
    node-z/.style={circle, fill=purple!85!black, inner sep=2.5pt, draw=black},
    edge/.style={thick, color=black!80},
    light-edge/.style={thick, color=black!30},
    cut-edge/.style={thick, dashed, color=red!75!black},
    no-edge/.style={thick, dashed, color=red!70!black}
]
    \node[node-v, label={below left:\large $v$}] (v) at (0, 0) {};
    \node[node-v, label={below right:\large $u$}] (u) at (8.0, 0) {};
    
    \draw[edge, line width=1.2pt] (v) -- (u);

    \draw[rounded corners=16pt, fill=blue!4, draw=blue!70!black, thick] 
        (0.6, 1.4) rectangle (6.4, 5.3);
    
    \node[blue!85!black, font=\bfseries\small, anchor=north east] at (6.2, 5.15) {$|U| = 7$};

    \draw[rounded corners=12pt, fill=blue!10, draw=blue!80!black, densely dashed] 
        (0.8, 1.8) rectangle (4.5, 5.0);
    
    \node[blue!90!black, font=\footnotesize\bfseries, anchor=north] at (3.8, 5.0) 
        {$|D| \geq 5$};

    \node[vertex, label={above:\scriptsize $x_1$}] (x1) at (2.2, 3.9) {};
    \node[vertex, label={above:\scriptsize $x_2$}] (x2) at (3.8, 3.9) {};
    \node[vertex, label={below:\scriptsize $x_3$}] (x3) at (2.2, 2.3) {};
    \node[vertex, label={below:\scriptsize $x_4$}] (x4) at (3.8, 2.3) {};
    \node[vertex, label={right:\scriptsize $x_5$}] (x5) at (5.4, 3.6) {};
    \node[vertex, label={right:\scriptsize $x_6$}] (x6) at (5.4, 2.6) {};

    \node[node-w, label={above:\large $w$}] (w) at (1.2, 3.1) {};

    \node[node-z, label={left:\large $z_w$}] (zw) at (-2.2, 4.0) {};

    \draw[edge, line width=1.5pt, color=purple!85!black] (w) -- (zw);

    \draw[edge, color=black!40] (v) -- (w);
    \draw[edge, color=black!40] (v) -- (x1);
    \draw[edge, color=black!40] (v) -- (x2);
    \draw[edge, color=black!40] (v) -- (x3);
    \draw[edge, color=black!40] (v) -- (x4);
    \draw[edge, color=black!40] (v) -- (x5);
    \draw[edge, color=black!40] (v) -- (x6);

    \draw[edge, color=blue!70!black] (w) -- (x1);
    \draw[edge, color=blue!70!black] (w) -- (x2);
    \draw[edge, color=blue!70!black] (w) -- (x3);
    \draw[edge, color=blue!70!black] (w) -- (x4);
    \draw[edge, color=blue!70!black] (w) -- (x5);
    \draw[edge, color=blue!70!black] (w) -- (x6);

    \draw[edge, color=blue!70!black] (x1) -- (x2);
    \draw[edge, color=blue!70!black] (x1) -- (x3);
    \draw[edge, color=blue!70!black] (x1) -- (x4);
    \draw[edge, color=blue!70!black] (x1) -- (x5);
    \draw[edge, color=blue!70!black] (x1) -- (x6);
    \draw[edge, color=blue!70!black] (x2) -- (x3);
    \draw[edge, color=blue!70!black] (x2) -- (x4);
    \draw[edge, color=blue!70!black] (x2) -- (x5);
    \draw[edge, color=blue!70!black] (x2) -- (x6);
    \draw[edge, color=blue!70!black] (x3) -- (x4);
    \draw[edge, color=blue!70!black] (x3) -- (x5);
    \draw[edge, color=blue!70!black] (x3) -- (x6);

    \draw[edge, color=blue!70!black] (x4) -- (x5);
    \draw[edge, color=blue!70!black] (x4) -- (x6);

    \draw[dashed, line width=0.8pt, color=blue!70!black] (x5) -- (x6);

    \draw[cut-edge] (zw) to[out=35, in=155] (x1);
    \draw[cut-edge] (zw) to[out=50, in=160] (x2);

    \node[font=\small, color=red!75!black, anchor=south] at (-0.5, 3.85) 
        {$k_w$ edges};
    \draw[no-edge] (v) to[bend left=12] (zw);
    \node[red!80!black, font=\Large\bfseries] at (-1.54, 2.2) {$\times$};
    \node[font=\scriptsize, text=red!80!black, align=center, anchor=east] at (-1.6, 2.0) 
        {no edge\\$z_w \notin N(v)$};

    \draw[edge] (u) -- (5.7, 1.8);
    \node[font=\footnotesize, anchor=south, rotate=-37] at (7.2, 0.55) 
        {$c$ edges};

\end{tikzpicture}
\caption{Structure of the neighborhood $N(w) = \{v\} \sqcup (U \setminus \{w\}) \sqcup \{z_w\}$ in the proof of Lemma~\ref{no-edge-lemma}. The vertex $w \in D$ has degree $6$ in $U$, and its remaining neighbor $z_w$ lies outside $N(v)$.}
\label{fig:no-edge-lemma}
\end{figure}

	To determine the triangle-degree $K_3\deg(w)$, we count the edges in the induced subgraph $G[N(w)]$ (see Figure~\ref{fig:no-edge-lemma} for the structural decomposition). The neighborhood of $w$ partitions into three disjoint sets: 
    \[N(w) = \{v\} \sqcup (U \setminus \{w\}) \sqcup \{z_w\}.\] 
    Since $z_w \notin N(v)$, there is no edge between $v$ and $z_w$. The edges of $G[N(w)]$ decompose as follows:
	\begin{itemize}
		\item $6$ edges connecting $v$ to all of $U \setminus \{w\}$ (since $U \subset N(v)$);
		\item $|E(U)| - 6$ edges within $U \setminus \{w\}$ (since $\deg_{G[U]}(w)=6)$;
		\item $k_w - 1$ edges between $z_w$ and $U \setminus \{w\}$, where $k_w = |N(z_w) \cap U| \ge 1$ is the total number of neighbors of $z_w$ in $U$.
	\end{itemize}
    Summing these contributions, the terms of $6$ cancel, yielding
    \[K_3\deg(w) = |E(U)| + k_w - 1.\]
	
    Since $G$ is $K_3$-irregular, the values $K_3\deg(w)$ are pairwise distinct for all $w \in D$, which implies that the values $k_w$ are pairwise distinct. 
    
    Since $|D| \ge 5$ and each $k_w \ge 1$, their sum is bounded from below by the sum of the five smallest positive integers:
	\[
	\sum_{w \in D} k_w \ge 1 + 2 + 3 + 4 + 5 = 15.
	\]
	
	The sum $\sum_{w \in D} k_w$ represents the number of edges from the set $\{z_w\}_{w \in D}$ to $U$, therefore the total number of edges between $U$ and $V(G) \setminus (U \cup \{u, v\})$ satisfies
    \[
    |E(U, V(G) \setminus (U \cup \{u, v\}))| \ge \sum_{w \in D} k_w \ge 15.
    \]
	
	On the other hand, we can bound the size of this edge cut directly from the degrees of vertices in $U$. Summing the degrees in $U$ (this sum is $7 \cdot 8 = 56$) and subtracting twice the internal $|E(U)|$ edges, the $7$ edges to $v$, and the $c$ edges to $u$, we obtain: 
	\[
    |E(U, V(G) \setminus (U \cup \{u, v\}))| = 56 - 2|E(U)| - 7 - c = 49 - 2|E(U)| - c.
    \]
    Since $|E(U)| \ge 20$ and $c \ge 0$, this cut contains at most $9$ edges. We thus obtain $15 \le |E(U, V(G) \setminus (U \cup \{u, v\}))| \le 9$, a contradiction.

    Notice that the contradiction above relies only on the inequality $|E(U)| \ge 20$. Now suppose that $K_3\deg(v) - K_3\deg(u) = 19$. Then $|E(U)| \ge K_3\deg(v) - K_3\deg(u) = 19$, and since $|E(U)| \ge 20$ is impossible, we must have $|E(U)| = 19$. As $\binom{7}{2} = 21$, the induced subgraph $G[U]$ lacks only two edges, which are incident to at most four vertices of $U$. Consequently, the set $D = \{w \in U : \deg_U(w) = 6\}$ satisfies $|D| \ge 7 - 4 = 3$. By the same edge decomposition of $G[N(w)]$ as above, each $w \in D$ has a unique neighbor $z_w \notin U \cup \{u, v\}$ and satisfies $K_3\deg(w) = |E(U)| + k_w - 1 = 18 + k_w$, where $k_w = |N(z_w) \cap U| \ge 1$. Since the vertices in $D$ have pairwise distinct $K_3$-degrees, the integers $k_w$ ($w \in D$) are pairwise distinct. As $|D| \ge 3$, at least one vertex $w \in D$ has $k_w \ge 3$, yielding $K_3\deg(w) \ge 21$.
\end{proof}
    
	Having bounded the difference in $K_3$-degrees for adjacent vertices, we now turn to the size of the common neighborhood $c = |N(u) \cap N(v)|$. The following lemma establishes a general upper bound on $c$, parametrized by the difference between their $K_3$-degrees.

	\begin{lemma}\label{intersection-lemma}
		Let $G$ be an $r$-regular $K_3$-irregular graph, and let $u, v \in V(G)$ be two vertices with $K_3$-degrees $d_1$ and $d_2$, respectively. If $c = |N(u) \cap N(v)|$, then		
		\[
			\binom{c}{2} + \ceil{\frac{r - c}{2}} \leq \binom{r}{2} - |d_1 - d_2|.
		\]
	\end{lemma}
	
	\begin{proof}
		Let $X = N(u) \cap N(v)$, $p = |E(X)|$, and $Y = N(u) \setminus X$. Without loss of generality, assume $d_1 \geq d_2$. We calculate the number of missing edges in $G[N(u)]$ in two ways. On the one hand, we know that $|E_{\overline{G}}(N(u))| = \binom{r}{2} - d_1$. On the other hand, $|E_{\overline{G}}(N(u))| = |E_{\overline{G}}(X)| + |E_{\overline{G}}(X, Y)| + |E_{\overline{G}}(Y)|$. Recall that $|E_{\overline{G}}(X)| = \binom{c}{2} - p$. Thus,
		\begin{equation}\label{eq1:edge-count}
		    \binom{r}{2} - d_1 = \binom{c}{2} - p + |E_{\overline{G}}(X, Y)| + |E_{\overline{G}}(Y)|.
		\end{equation}
		Note that for every $y \in Y$, we have $d_{\overline{G}[N(u)]}(y) \geq 1$. Indeed, if there were a vertex $y$ with $d_{\overline{G}[N(u)]}(y) = 0$, this would imply that $y$ and $u$ are true twins, contradicting Lemma~\ref{lemma-no-twins}. Therefore, $\sum_{y \in Y}d_{\overline{G}[N(u)]}(y) \geq r - c$. We also have 
		\[
		\sum_{y \in Y}d_{\overline{G}[N(u)]}(y) = |E_{\overline{G}}(X, Y)| + 2 \cdot |E_{\overline{G}}(Y)|.
		\] 
		Since 
		\[
		|E_{\overline{G}}(X, Y)| + |E_{\overline{G}}(Y)| \geq \frac{|E_{\overline{G}}(X, Y)| + 2 \cdot|E_{\overline{G}}(Y)|}{2} \geq \frac{r - c}{2},
		\]
        and since $|E_{\overline{G}}(X, Y)| + |E_{\overline{G}}(Y)|$ is a non-negative integer, we conclude that
        \[|E_{\overline{G}}(X, Y)| + |E_{\overline{G}}(Y)| \geq \ceil{\frac{r-c}{2}}.\]
		Substituting this into \eqref{eq1:edge-count} and rearranging gives
		\[
			\binom{c}{2} + \ceil{\frac{r - c}{2}} \leq \binom{r}{2} - d_1 + p.
		\]
		Finally, note that $d_2 \geq p$, which yields
		\[
			\binom{c}{2} + \ceil{\frac{r - c}{2}} \leq \binom{r}{2} - d_1 + d_2.
		\]
	\end{proof}

    Setting $r=8$ in Lemma~\ref{intersection-lemma} gives the following bounds.
    
	\begin{corollary}\label{intersection-corollary}
		Let $G$ be an $8$-regular $K_3$-irregular graph, and let $u, v \in V(G)$ be two vertices with $K_3$-degrees $d_1$ and $d_2$, respectively. Let $c = |N(u) \cap N(v)|$. Then:
		\begin{itemize}
			\item If $|d_1 - d_2| \geq 21$, then $c \leq 3$.
			\item If $17 \leq |d_1 - d_2| \leq 20$, then $c \leq 4$.
		\end{itemize}
	\end{corollary}

    While the preceding results establish strict upper bounds on the size of the common neighborhood, the next lemma provides a structural lower bound. Unlike the previous results, this bound is formulated for pairs of non-adjacent vertices, and also depends on the order of a graph.

    \begin{lemma}\label{intersection-lower-bound-lemma}
		Let $G$ be an $r$-regular $K_3$-irregular graph on $n$ vertices. Let $u, v \in V(G)$ be non-adjacent vertices with $K_3$-degrees $d_1$ and $d_2$, respectively. If $c = |N(u) \cap N(v)|$, then
		\[
			2|d_1 - d_2| \le (r-c)(n - 2r - 2 + c) + c \cdot \min \{r-2, r-c\}.
		\]
	\end{lemma}
	
	\begin{proof}
		Without loss of generality, assume $d_1 \ge d_2$. Let $X = N(u) \cap N(v)$, $Y_u = N(u) \setminus X$, and $Y_v = N(v) \setminus X$. Let $B = V(G) \setminus (\{u, v\} \cup N(u) \cup N(v))$. Note that $|X| = c$ and $|Y_u| = |Y_v| = r - c$, so $|B| = n - 2r - 2 + c$.
		
		The $K_3$-degree of $u$ equals the number of edges in $G[N(u)]$, so $d_1 = |E(X)| + |E(Y_u)| + |E(X, Y_u)|$. Let us sum the degrees of the vertices in $Y_u$. Since $G$ is $r$-regular and each $y \in Y_u$ is adjacent to $u$ but not to $v$, there are $(r-1)|Y_u|$ edges connecting $Y_u$ to $V(G) \setminus \{u\}$. These edges can go to $X, Y_u, Y_v$, or $B$. Note that each edge within $Y_u$ contributes $2$ to the sum of degrees. Thus,
		\[
			|E(X, Y_u)| + 2|E(Y_u)| + |E(Y_u, Y_v)| + |E(Y_u, B)| = (r-1)(r-c).
		\]
		We can express $2|E(Y_u)|$ from the equation for $d_1$ as $2d_1 - 2|E(X)| - 2|E(X, Y_u)|$. Substituting this into the above gives
		\[
			2d_1 - 2|E(X)| - |E(X, Y_u)| + |E(Y_u, Y_v)| + |E(Y_u, B)| = (r-1)(r-c).
		\]
		By a symmetric argument for $v$ and $Y_v$, we have
		\[
			2d_2 - 2|E(X)| - |E(X, Y_v)| + |E(Y_u, Y_v)| + |E(Y_v, B)| = (r-1)(r-c).
		\]
		Subtracting the second equation from the first yields
		\[
			2(d_1 - d_2) = |E(Y_v, B)| - |E(Y_u, B)| + |E(X, Y_u)| - |E(X, Y_v)|.
		\]
		Since $|E(Y_u, B)| \ge 0$ and $|E(X, Y_v)| \ge 0$, we obtain the inequality
		\[
			2(d_1 - d_2) \le |E(Y_v, B)| + |E(X, Y_u)|.
		\]
		Now we find trivial upper bounds for the terms on the right-hand side. The number of edges between the disjoint sets $Y_v$ and $B$ is at most the product of their sizes, so $|E(Y_v, B)| \le |Y_v| \cdot |B| = (r-c)(n - 2r - 2 + c)$. 
		For the edges between $X$ and $Y_u$, observe that each of the $c$ vertices in $X$ is adjacent to both $u$ and $v$, leaving at most $r-2$ edges for the rest of the graph. At the same time, it can send at most $|Y_u| = r-c$ edges to $Y_u$. Thus, each vertex in $X$ sends at most $\min\{r-2, r-c\}$ edges to $Y_u$, which gives $|E(X, Y_u)| \le c \cdot \min(r-2, r-c)$. 
		
		Substituting these bounds yields the desired inequality.
	\end{proof}

    Setting $r=8$ in Lemma~\ref{intersection-lower-bound-lemma} gives the following bounds for $n=21$ and $n=22$.
    
	\begin{corollary}\label{intersection-lower-bound-corollary}
		Let $G$ be an $8$-regular $K_3$-irregular graph on $n$ vertices. Then, we have the following lower bounds on the cardinality of the common neighborhood $c = |N(u) \cap N(v)|$ for non-adjacent vertices $u, v$ with $K_3$-degrees $d_1, d_2$:
		\begin{itemize}
			\item If $n = 22$ and $|d_1 - d_2| \ge 21$, then $c \ge 2$.
			\item If $n = 21$ and $|d_1 - d_2| \ge 20$, then $c \ge 2$.
		\end{itemize}
	\end{corollary}

    The structural properties established throughout this section provide concrete restrictions on candidate graphs. In particular, vertices with large $K_3$-degree differences cannot be adjacent (Lemma~\ref{no-edge-lemma}), and the size of their common neighborhood is bounded from above and below (Corollaries~\ref{intersection-corollary} and~\ref{intersection-lower-bound-corollary}). In the next section, we incorporate these properties directly into an integer linear programming formulation, using them to prune the search space during the computational search. 
    
	\section{Integer linear programming formulation}\label{sec:ilp-formulation}
	
	In this section we translate the problem of finding an $r$-regular $K_3$-irregular graph into an integer linear programming (ILP) feasibility problem. We encode the adjacency matrix, vertex degrees, triangle-degrees, and all structural constraints as linear equalities and inequalities over binary and integer variables.\footnote{Strictly speaking, the model is a mixed-integer linear program (MILP), combining binary variables (for edges and triangle indicators) with general integer variables (for $K_3$-degrees). Since every binary variable is integer-valued, we simply write ILP.}
	
	Since we are interested in feasibility rather than optimization, our ILP has no meaningful objective function; we use the trivial objective ``minimize zero''. A feasible solution corresponds to an $r$-regular $K_3$-irregular graph, while a proof of infeasibility certifies that no such graph exists for the given parameters.
	
	We wrote a \texttt{C++} program that generates a \texttt{.lp} file for the given set of parameters (order~$n$, regularity~$r$, and the admissible range of $K_3$-degrees). While \texttt{.lp} is a standard format accepted by most solvers, we use the SCIP Optimization Suite~\cite{SCIP} for all computations in this paper. The source code, generated \texttt{.lp} files, and solver logs are available in the accompanying repository~\cite{IrregularILP}.
	
	\subsection{Encoding a graph}
	
	Let $G$ be a graph with vertex set $V(G) = \{w_0, w_1, \ldots, w_{n-1}\}$ and edge set $E(G)$. 
	Since we consider simple undirected graphs, only the upper triangle $i < j$ of the adjacency matrix is needed to define $G$. 
	For each two-element subset $\{i,j\} \subseteq \ran{0}{n-1}$, we introduce a binary variable $x_{ij} \in \{0,1\}$ such that 
	\[x_{ij} = 1 \iff \{w_i, w_j\} \in E(G).\]
	Since indices form an unordered pair (and to simplify notation), $x_{ij}$ and $x_{ji}$ denote the same variable. This is our starting point; we use these edge variables to express all the required structural constraints as linear equations and inequalities. 
	
	\subsection{Regularity}
	
	To enforce that $G$ is $r$-regular, we require that every vertex has degree exactly~$r$:
	\[\sum_{\substack{j=0 \\ j \ne i}}^{n-1} x_{ij} = r, \qquad \text{for all } i \in \ran{0}{n-1}.\]
	In our ILP model, this yields exactly $n$ linear constraints (one per vertex).
	
	\subsection{Counting triangles}
	
	Encoding triangle-degrees is more involved. First, we introduce a binary indicator for every potential triangle: it equals~$1$ if and only if the corresponding three vertices form an induced $K_3$ subgraph in~$G$. Then, for each vertex, we sum the indicators of all triangles it participates in to obtain its $K_3$-degree.
	
	For each three-element subset $\{i, j, k\} \subseteq \{0, \ldots, n-1\}$, we introduce a binary variable
	\[t_{ijk} = 1 \iff G[\{w_i, w_j, w_k\}] \cong K_3.\]
	As with the edge variables, the indices of $t_{ijk}$ form an unordered set.
	A triangle on $w_i, w_j, w_k$ exists if and only if all three edges are present. We linearize the product $t_{ijk} = x_{ij} \cdot x_{ik} \cdot x_{jk}$ using the standard McCormick inequalities:
	\begin{align*}
		t_{ijk} &\le x_{ij},\\
		t_{ijk} &\le x_{ik},\\
		t_{ijk} &\le x_{jk},\\
		t_{ijk} &\ge x_{ij} + x_{ik} + x_{jk} - 2.
	\end{align*}
	
	For each vertex $w_i$, we introduce a general integer variable $d_i \ge 0$ representing its $K_3$-degree:
	\[d_i = \sum_{\substack{\{j, k\} \subset \ran{0}{n-1} \setminus \{i\}}} t_{ijk}, \qquad \text{for all } i \in \ran{0}{n-1}.\]
	
	We also introduce a general integer variable $T$ representing the total number of triangles in the graph. By Corollary~\ref{triangle-shake-lemma}, the sum of all $K_3$-degrees equals three times the total number of triangles:
	\[\sum_{i=0}^{n-1} d_i = 3T.\]
	This constraint naturally prunes search branches where the sum of $K_3$-degrees is not divisible by $3$. The bounds on $T$ are derived from the admissible range of $K_3$-degrees: 
	$T_{\min} = \ceil{\frac{\sum_{k=0}^{n-1} k + n \cdot \ell}{3}}$ and 
	$T_{\max} = \floor{\frac{n \cdot u - \sum_{k=0}^{n-1} k}{3}}$, where $\ell$ and $u$ are the minimum and maximum values in the  admissible range of triangle-degrees.
	
	\subsection{No true twins condition}
	
	By the twin-free property of $K_3$-irregular graphs (Lemma~\ref{lemma-no-twins}), the number of common neighbors of any pair of distinct vertices $w_i$ and $w_j$ is at most $r-2$. Since every common neighbor of $w_i$ and $w_j$ forms a triangle with them, we obtain the following upper bound for each $\{i, j\} \subseteq \ran{0}{n-1}$:
	\[\sum_{\substack{k \in \ran{0}{n-1}} \setminus \{i,j\}} t_{ijk} \le r-2.\]
	
	\subsection{Triangle irregularity}
	
	The key property of $K_3$-irregularity requires that all vertices have pairwise distinct $K_3$-degrees. If all vertices were a priori indistinguishable, the most efficient way to break symmetry and enforce distinctness would be to impose a global total ordering:
	\[d_0 < d_1 < \cdots < d_{n-1}, \qquad \text{i.e., } d_i - d_{i+1} \le -1 \text{ for } i=0, \ldots, n-2.\]
	
	However, the partitioning technique (see Section~\ref{sec:part-tech}) breaks this global symmetry by dividing the vertices into two distinct sets, $A$ and $B$. While the vertices within $A$ are indistinguishable from each other, as are vertices in $B$, a vertex in $A$ is structurally different from a vertex in $B$ (since it is adjacent to the anchor). Therefore, we cannot impose a global total ordering without potentially excluding valid graphs.
	
	Instead, we impose a total ordering strictly within $A$ and strictly within $B$. To enforce distinctness between any two vertices $w_i \in A$ and $w_j \in B$, we use the Big-$M$ method. For each such pair of vertices, we introduce a binary auxiliary variable $b_{ij} \in \{0,1\}$ and add the inequalities:
	\begin{align*}
		d_i - d_j - M \cdot b_{ij} &\le -1,\\
		d_j - d_i + M \cdot b_{ij} &\le M-1,
	\end{align*}
	where $M = u - \ell + 1$ (given the admissible range of $K_3$-degrees $\ran{\ell}{u}$), it is a sufficiently large constant that bounds the maximum possible difference between any two admissible $K_3$-degrees. The variable $b_{ij}$ acts as a switch: when $b_{ij} = 0$, the constraints force $d_i < d_j$; when $b_{ij} = 1$, they force $d_j < d_i$. In either case, $d_i \ne d_j$ is guaranteed. Finally, the bounds on each $K_3$-degree variable $d_i$ are set according to the current admissible range.
	
	Although replacing a global total ordering with pairwise Big-$M$ constraints generally makes the model computationally harder to solve, the massive reduction in the search space gained by fixing the anchor vertex (and later a second vertex) far outweighs this performance cost.
	
	\subsection{Fix one vertex --- the anchor}\label{sec:anchor}
	
	To exploit the structural decomposition introduced in Section~\ref{sec:part-tech}, we fix an \emph{anchor} vertex $w_0$ with a prescribed $K_3$-degree $d_0 = d$, and we partition the remaining vertices into $A = N(w_0) = \{w_1, \ldots, w_r\}$ and $B = V \setminus N[w_0] = \{w_{r+1}, \ldots, w_{n-1}\}$.
	
	This partition yields the following constraints:
	\begin{itemize}
		\item First, we fix the edge indicators incident to the anchor:
		\begin{align*}
			x_{0j} &= 1, \quad \text{for } 1 \le j \le r,\\
			x_{0j} &= 0, \quad \text{for } r+1 \le j \le n - 1.
		\end{align*}
		
		\item Second, we fix the exact number of edges within and between the parts:
		\begin{align*}
			\sum_{1 \le i < j \le r} x_{ij} &= d,\\
			\sum_{i=1}^{r} \sum_{j=r+1}^{n-1} x_{ij} &= r(r-1) - 2d,\\
			\sum_{r+1 \le i < j \le n-1} x_{ij} &= \tfrac{nr}{2} - r^2 + d.
		\end{align*}
	\end{itemize}
	In the special case $d = 0$, the induced subgraph $G[A]$ contains no edges, so we additionally fix $x_{ij} = 0$ for all $1 \le i < j \le r$. 
	
	Even with this machinery, we are already able to show that no $8$-regular $K_3$-irregular graphs exist for $17 \le n \le 19$. For harder cases, we will introduce additional structural constraints in Section~\ref{sec:fix-second}.

	\subsection{ILP verification}
	
	To ensure the correctness of our \texttt{C++} generator, we performed sanity checks using known valid graphs~\cite{ZenodoK3Irregular, Hak:25, reg-triangle:24}. For example, we generated the ILP instances -- typically for the smallest known example with $n=24$ and $r=9$ -- and provided the corresponding graph to the SCIP solver as an initial feasible solution (via an \texttt{.mst} file). To accomplish this, we wrote a function within the same project that takes a graph and outputs the values of its edge indicators. This process requires care: we reindexed the vertices of the known graph to match the specific constraints of the generated instance (such as the anchor vertex). Depending on the test, we assigned the appropriate vertex (with the given $K_3$-degree) to be the anchor $w_0 = v_d$, its neighbors to be $N(v_d) = \{w_1, \ldots w_r\}$, and the non-neighbors to $w_{r+1}, \ldots w_{n-1}$, exactly mirroring the expected structure. The solver successfully verifies the feasibility of this solution, confirming that our generated constraints correctly encode the problem. This approach helped to identify and fix bugs in the code early during the development stage. We used this same approach to validate the additional structural constraints introduced later in this paper.
	
	\section{Computation search strategy and results}\label{sec:search-results}

	All computations were performed on a Google Cloud Platform \texttt{c4-standard-8} virtual machine instance equipped with an Intel Granite Rapids processor. We configured the instance to use $4$ physical cores without hyperthreading and $30$~GB of RAM, running Debian 13 OS. The ILP models were solved using the SCIP Optimization Suite version~10.0.3. To ensure consistent benchmarking and prevent CPU context switching overhead, we used the \texttt{taskset} utility to pin each solver process separately to a dedicated physical core.
    
	The general strategy for each value of $n$ is as follows. We start with the full admissible range $\ran{0}{22}$ of $K_3$-degrees, as established in Proposition~\ref{8-reg-upper-bound}, and iteratively reduce this range from its boundaries.
	
	At each step, we perform the reduction by showing that the current maximum $u$ or minimum $\ell$ in the current range cannot exist in the graph. Thus, from a range $\ran{\ell}{u}$, we transition to either $\ran{\ell}{u-1}$ or $\ran{\ell+1}{u}$.
	
    To eliminate a candidate value~$d$, we designate a vertex with $K_3$-degree $d$ as the \emph{anchor} vertex, encode the resulting instance as an ILP, and verify infeasibility using the SCIP solver.
	This process terminates when the remaining range $\ran{\ell'}{u'}$ meets one of the following two stopping conditions:
	\begin{enumerate}
		\item \textbf{Pigeonhole principle:} The size of the range is strictly less than $n$ (i.e., $|\ran{\ell'}{u'}| < n$), making it impossible to assign $n$ distinct $K_3$-degrees to the vertices;
		\item \textbf{Divisibility violation:} The range is of size $n$ (i.e., $|\ran{\ell'}{u'}| = n$), but the sum of its values is not divisible by three, which contradicts Corollary~\ref{triangle-shake-lemma}.
	\end{enumerate}
	
	\subsection[Initial eliminations: n = 17, 18, 19]{Initial eliminations: $\boldsymbol{n = 17, 18, 19}$}
	
	In this subsection, we resolve the three smallest possible orders. The results of the ILP solver runs, including branch-and-bound tree statistics (maximum depth and node count), simplex iterations (LP iter), and runtimes, are summarized in Table~\ref{table-17-18-19}.
	
	\medskip
	\noindent\textbf{Case $\boldsymbol{n=17}$.}
	We start with the range $\ran{0}{22}$, and eliminate values from the top in sequence: $22$, $21$, $20$, $19$, $18$, and $17$. Each step fixes the current maximum as the anchor, and the solver confirms infeasibility very rapidly (in seconds for the easier cases and in less than three minutes for the hardest case). This reduces the range to $\ran{0}{16}$, which contains exactly $17=n$ values, and the sum $\sum_{k=0}^{16}k = 136$ is not divisible by three, so by Corollary~\ref{triangle-shake-lemma} no such graph exists.
	
	\medskip
	\noindent\textbf{Case $\boldsymbol{n=18}$.}
	Starting from $\ran{0}{22}$, we again eliminate extreme values from the top in sequence: $22$, $21$, $20$, $19$, $18$, and $17$. The computational effort increases noticeably as we reach tighter ranges (for instance, eliminating $18$ takes over an hour and explores hundreds of thousands of nodes). Thus, the range is reduced to $\ran{0}{16}$. It contains $17$ values, which triggers the pigeonhole principle stopping condition.
	
	\medskip
	\noindent\textbf{Case $\boldsymbol{n=19}$.}
	Starting from $\ran{0}{22}$, we eliminate $22$, $21$, and $20$ from the top, reducing the admissible range to $\ran{0}{19}$. At this point, attempting to eliminate the current maximum $19$ proved computationally expensive. Instead, we pivot to eliminate the minimum value~$0$, obtaining $\ran{1}{19}$. This range contains exactly $19$ values, and the sum $\sum_{k=1}^{19}k = 190$ is not divisible by three, which contradicts Corollary~\ref{triangle-shake-lemma}.
	Had we eliminated $19$ instead, the resulting range $\ran{0}{18}$ would contain $19$ values with a sum of $171$, which is divisible by three, thus requiring one more step.
    
    As seen in Table~\ref{table-17-18-19}, while the fix one-vertex strategy successfully resolves $n \in \{17, 18, 19\}$, the exponential growth in branch-and-bound nodes and runtime makes individual eliminations prohibitively slow for larger values of $n$. We report these standard solver statistics (maximum tree depth, processed nodes, LP iterations, and runtime) not as a formal algorithmic benchmark, but merely to illustrate the steep escalation in computational effort required as the admissible range narrows. This computational bottleneck directly motivates the advanced symmetry-breaking techniques introduced in the next subsection to process the remaining orders $n \in \{20, 21, 22\}$.
    \begin{table}[ht]
        \centering
        \begin{tabular}{@{}lcrrr@{}}
        \toprule
            Evaluated range & Max Depth & Nodes & LP iter & Time \\
        \midrule
            \multicolumn{5}{@{}l}{\textbf{Case $\boldsymbol{n=17}$}} \\
            $\{0, \ldots, 21, \mathbf{\underline{22}}\}$ & 0 & 1 & 1.2k & 0.12\,s \\
            $\{0, \ldots, 20, \mathbf{\underline{21}}\}$ & 0 & 1 & 1.9k & 0.24\,s \\
            $\{0, \ldots, 19, \mathbf{\underline{20}}\}$ & 13 & 339 & 46.2k & 4.61\,s \\
            $\{0, \ldots, 18, \mathbf{\underline{19}}\}$ & 22 & 1.6k & 292k & 19.66\,s \\
            $\{0, \ldots, 17, \mathbf{\underline{18}}\}$ & 28 & 5.3k & 960k & 63.59\,s \\
            $\{0, \ldots, 16, \mathbf{\underline{17}}\}^{[\text{stop}]}$ & 30 & 14.6k & 2.7M & 165.32\,s \\
        \midrule
            \multicolumn{5}{@{}l}{\textbf{Case $\boldsymbol{n=18}$}} \\
            $\{0, \ldots, 21, \mathbf{\underline{22}}\}$ & 0 & 1 & 1.7k & 0.18\,s \\
            $\{0, \ldots, 20, \mathbf{\underline{21}}\}$ & 13 & 220 & 21.8k & 3.38\,s \\
            $\{0, \ldots, 19, \mathbf{\underline{20}}\}$ & 28 & 7.8k & 1.1M & 86.19\,s \\
            $\{0, \ldots, 18, \mathbf{\underline{19}}\}$ & 26 & 6.3k & 1.1M & 80.38\,s \\
            $\{0, \ldots, 17, \mathbf{\underline{18}}\}$ & 46 & 242k & 64.5M & 74\,m \\
            $\{0, \ldots, 16, \mathbf{\underline{17}}\}^{[\text{stop}]}$ & 42 & 249k & 63.2M & 69\,m \\
        \midrule
            \multicolumn{5}{@{}l}{\textbf{Case $\boldsymbol{n=19}$}} \\
            $\{0, \ldots, 21, \mathbf{\underline{22}}\}$ & 27 & 4.1k & 740k & 63.75\,s \\
            $\{0, \ldots, 20, \mathbf{\underline{21}}\}$ & 56 & 320k & 76.0M & 96\,m \\
            $\{0, \ldots, 19, \mathbf{\underline{20}}\}$ & 62 & 2.16M & 587M & 12\,h \\
            $\{\mathbf{\underline{0}}, 1, \ldots, 19\}^{[\text{stop}]}$ & 40 & 96.1k & 46.7M & 51\,m \\
        \bottomrule
        \end{tabular}
        \caption{Computational results for $17 \le n \le 19$.}\label{table-17-18-19}
    \end{table}
    
	\subsection{Fixing a second vertex in the ILP}\label{sec:fix-second}
	
	For $n \ge 20$, individual eliminations require prohibitively long runtimes. To overcome this, we adopt a more sophisticated strategy: we embed the structural restrictions derived in Section~\ref{sec:advanced-props} as constraints in the ILP model, allowing us to split the hardest computations into smaller subproblems.
	
	The core technique is to \emph{fix two vertices} simultaneously---specifically, those realizing the current maximum $K_3$-degree ($u$) and minimum $K_3$-degree ($\ell$). This allows us to fix more structure and impose more conditions on the neighborhoods of the two vertices, which significantly reduces the solver's search tree.
    
	However, there is a trade-off: when the solver proves infeasibility with two vertices fixed, we merely show that these two specific vertices \emph{cannot coexist} in the graph. We cannot eliminate both values at once. Instead, we must branch our computation into two independent subproblems: one where $u$ is eliminated ($\ran{\ell}{u-1}$), and another where $\ell$ is eliminated ($\ran{\ell+1}{u}$). Both branches must then be processed further.

    Before encoding the second vertex, we must first restrict the possible ways these two fixed vertices can be positioned in the graph. By Lemma~\ref{no-edge-lemma}, the extreme vertices $v_\ell$ and $v_u$ cannot be adjacent whenever $u - \ell \ge 20$, nor when $u - \ell = 19$ provided $u \le 20$. As this holds for all extreme pairs evaluated below, we can always safely place the bottom-most vertex $v_\ell$ in the non-neighborhood part $B$ of the top-most vertex $v_u$. We usually anchor the top-most vertex as the anchor (from our experiments, computation is faster in these settings). 

    Further, by Lemma~\ref{intersection-lemma} we impose constraints on the neighborhood of $v_\ell$, which is split between $A$ and $B$. Notice that $N(v_\ell) \cap A$ is exactly the common neighborhood $N(v_\ell) \cap N(v_u)$. We denote the size of this common neighborhood by $c = |N(v_u) \cap N(v_\ell)|$. By Corollary~\ref{intersection-corollary} and Corollary~\ref{intersection-lower-bound-corollary}, $c$ is strictly bounded: $c_{\min} \le c \le c_{\max}$.
    Consequently, the number of neighbors of $v_\ell$ in $B$ is exactly $r - c$, which gives a strict lower bound $|N(v_\ell) \cap B| \ge r - c_{\max}$ (see Figure~\ref{fig:part_bounds}).
    
	It is important to emphasize that imposing the lower bound $|N(v_\ell) \cap B|$ is a completely general constraint that covers all valid configurations. In this phase, we do not strictly fix the \emph{exact} number of edges from $v_\ell$ to $A$ or $B$. Instead, the ILP solver is free to distribute the remaining $c_{\max}$ neighbors of $v_\ell$ between $A$ and $B$ to satisfy the structural conditions. This approach ensures no potential solutions are missed while significantly pruning the search space. We call this type a \emph{relaxed} formulation.
	
	\begin{figure}[ht]
		\begin{tikzpicture}[auto,on grid, state/.style ={circle, top color =black , bottom color = black , draw, minimum width=1mm}, inner sep=1.5pt, scale = 1]
			
			\node[state, label={[xshift=-5pt]left:$v_u$}] (vu) at (0, 0) {};
			
			\draw[line width=0.5mm] (4, 0) ellipse (2cm and 2.5cm);
			\node[below=1.25cm of 1] at (4, 4.5) {\large $A$};
			
			\draw[line width=0.5mm] (13, 0) ellipse (2cm and 2.5cm);
			\node[below=1.25cm of 1] at (13, 4.5) {\large $B$};
			
			\draw[line width=0.5mm] (5.95, -0.5) -- (11.05, -0.5);
			\node[] at (8.5, 0) {\large $\ldots$};
			\draw[line width=0.5mm] (5.95, 0.5) -- (11.05, 0.5);
			
			\coordinate (T1) at (2.75, 1.95);
			\coordinate (T2) at (2.75, -1.95);
			\draw[line width=0.5mm] (vu) -- (T1);
			\draw[line width=0.5mm] (vu) -- (T2);
			
			\node[state, label=above:$v_\ell$] (vl) at (13, 1.7) {};
			
			\draw[line width=0.5mm, dashed] (vl) -- (12.2, 0.7);
			\draw[line width=0.5mm, dashed] (vl) -- (12.7, 0.5);
			\draw[line width=0.5mm, dashed] (vl) -- (13.3, 0.5);
			\draw[line width=0.5mm, dashed] (vl) -- (13.9, 0.7);
            \draw[line width=0.2mm, dashed] (13, 0.4) ellipse (1.3cm and 0.75cm);
			\node at (13.0, 0.1) {\small $r - c_{\max}$};
			
		\end{tikzpicture}
		\caption{Refined partitioning with a second fixed vertex $v_\ell \in B$. Its neighborhood is distributed between $A$ and $B$, constrained by the common neighborhood size $c = |N(v_\ell) \cap N(v_u)| \leq c_{\max}$, obtained by Corollary~\ref{intersection-corollary}. Consequently, we have a lower bound on the neighbors in $B$: $|N(v_\ell) \cap B| = r - c \ge r - c_{\max}$.}\label{fig:part_bounds}
	\end{figure}
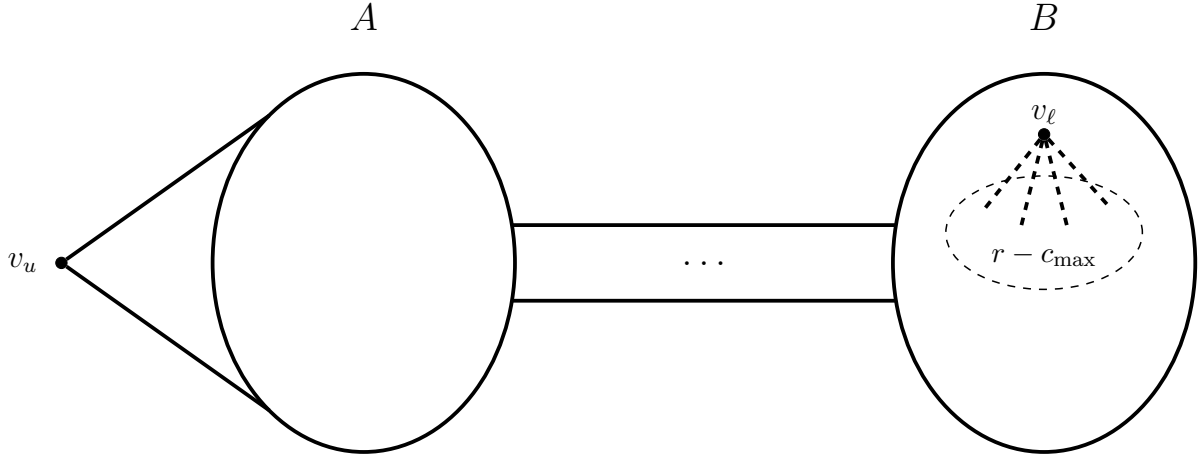
	
	For the hardest cases, we take this structural refinement a step further. Instead of relying solely on the relaxed lower bound constraint, we branch the ILP problem into several disjoint subproblems by fixing the \emph{exact} value of $c$. This dictates that $v_\ell$ has exactly $c$ neighbors in $A$ and exactly $r - c$ neighbors in~$B$. 
    
    Now, the set $N(v_\ell)$ is strictly fixed in the entire graph, see Figure~\ref{fig:part_exact}, naturally inducing a disjoint partition of both $A$ and $B \setminus \{v_\ell\}$ into two parts based on adjacency to $v_\ell$. 
	
	\begin{figure}[ht]
		\begin{tikzpicture}[auto,on grid, state/.style ={circle, top color =black , bottom color = black , draw, minimum width=1mm}, inner sep=1.5pt, scale = 1]
			
			\node[state, label={[xshift=-5pt]left:$v_u$}] (vu) at (0, 0) {};
			
			\draw[dashed] (4, 0) ellipse (2cm and 2.5cm);
			\node[below=1.25cm of 1] at (4, 4.5) {\large $A$};
			
			\draw[line width=0.3mm] (4, 1.1) ellipse (1.5cm and 0.8cm);
			\node at (4, 1.35) {\small $N(v_\ell) \cap A$:};
			\node at (4, 0.8) {\footnotesize $c$ vertices};
			
			\draw[line width=0.3mm] (4, -1.1) ellipse (1.5cm and 0.8cm);
			\node at (4, -0.85) {\small $A \setminus N(v_\ell)$:};
			\node at (4, -1.4) {\footnotesize $r-c$ vert.};
			
			\draw[dashed] (13, 0) ellipse (2cm and 2.5cm);
			\node[below=1.25cm of 1] at (13, 4.5) {\large $B$};
			
			\draw[line width=0.3mm] (13, 0.6) ellipse (1.5cm and 0.8cm);
			\node at (13, 0.85) {\small $N(v_\ell) \cap B$:};
			\node at (13, 0.3) {\footnotesize $r-c$ vert.};
			
			\draw[line width=0.3mm] (13, -1.2) ellipse (1.5cm and 0.8cm);
			\node at (13, -0.85) {\small $B \setminus N[v_\ell]$:};
			\node[align=center] at (13, -1.55) {\footnotesize $n-2r-2+c$ \\[-1ex] \footnotesize vert.};
			 
			\draw[line width=0.5mm] (5.95, -0.5) -- (11.05, -0.5);
			\node[] at (8.5, 0) {\large $\ldots$};
			\draw[line width=0.5mm] (5.95, 0.5) -- (11.05, 0.5);
			
			\draw[line width=0.5mm] (vu) -- (2.75, 1.1);
			\draw[line width=0.5mm] (vu) -- (2.75, -1.1);
			
			\node[state, label=above:$v_\ell$] (vl) at (13, 1.9) {};
			
			\draw[line width=0.5mm] (vl) -- (5.3, 1.1);
			\draw[line width=0.5mm] (vl) -- (13, 1.2);
			
		\end{tikzpicture}
		\caption{Further refinement by fixing the exact common neighborhood size $c$. Both $A$ and $B \setminus \{v_\ell\}$ are each disjointly partitioned into two parts according to their adjacency to $v_\ell$.}\label{fig:part_exact}
	\end{figure}

    While this approach requires launching the ILP solver separately for each exact value of $c$, it significantly reduces the solver's search tree for each individual run, making the hardest cases computationally tractable.
    
	We now describe how the second fixed vertex $v_\ell$ is encoded in the ILP. In this paper we use the pairs of fixed vertices such that $v_uv_\ell \notin E(G)$, thus $v_\ell \in B$ (nevertheless, our generator~\cite{IrregularILP} supports fixing two adjacent vertices).
	
	The neighborhood of $v_\ell$ is distributed between $A$ and $B$. Let $c = |N(v_\ell) \cap A|$, so that $|N(v_\ell) \cap B| = r - c$. By Corollary~\ref{intersection-corollary} and Corollary~\ref{intersection-lower-bound-corollary}, we have explicit bounds for the common neighborhood $c = |N(v_\ell) \cap N(v_u)|$.
	
	In the \emph{relaxed} formulation, we only impose this lower bound: we fix $(r - c_{\max})$ specific vertices in $B$ as neighbors of $v_\ell$ and let the solver distribute the remaining $c_{\max}$ neighbors freely between $A$ and $B$:
	
	\begin{align*}
		x_{v_\ell,\, w} &=1, \quad \text{for each } w \in \{r + 2, \ldots, r + 1 + (r - c_{\max})\}.
	\end{align*}
	
	\[\sum_{\{w, w'\} \subseteq \{r+2, \ldots, r + 1 + (r - c_{\max})\}} x_{ww'} \le \ell.\]
	
	In the \emph{exact} formulation (used for the hardest subcases), we fix the precise value $c = |N(v_\ell) \cap A|$ and $r - c$ neighbors in $B$. This adds the following constraints:
	\begin{align*}
		x_{v_\ell,\, w} &=1, \quad \text{for each } w \in N(v_\ell), \\
		x_{v_\ell,\, w} &=0, \quad \text{for each } w \in V(G) \setminus N[v_\ell].
	\end{align*}
	
	The $K_3$-degree of $v_\ell$ equals the number of edges among its neighbors:
	\[\sum_{\{w, w'\} \subseteq N(v_\ell)} x_{ww'} = \ell.\]
	In particular, when $\ell = 0$, all neighbors of $v_\ell$ are pairwise non-adjacent: $x_{ww'}=0$ for all $\{w,w'\} \subseteq N(v_\ell)$.
	
	Fixing $v_\ell$ further refines the symmetry-breaking strategy. The sets $A$ and $B$ each split into two subgroups: those adjacent to $v_\ell$ and those not. We impose separate total orderings within each of the four resulting subgroups, and use pairwise Big-$M$ constraints between subgroups.
	
	\subsection[Final eliminations n = 20, 21, 22]{Final eliminations $\boldsymbol{n = 20, 21, 22}$}
    
	\medskip
	\noindent\textbf{Case $\boldsymbol{n=20}$.}
    The computational results and search tree statistics for this case are summarized in Table~\ref{table-20} and the subcase elimination DAG is illustrated in Figure~\ref{fig:tree_n20}. Starting from the full admissible range $\ran{0}{22}$, we first successfully eliminate $22$ using the standard single-vertex anchoring strategy, reducing the range to $\ran{0}{21}$.

    At this stage, single-vertex eliminations become intractable (the previous step with $22$ took about 93 hours). We pivot to the two-vertex strategy by simultaneously fixing the new maximum ($21$) and minimum ($0$). By Lemma~\ref{no-edge-lemma}, these extreme vertices are non-adjacent. We anchor $v_{21}$ and place $v_0$ in the non-neighborhood part~$B$. Infeasibility of this configuration branches the problem into two subcases:
	
	\begin{itemize}
		\item \emph{Subcase $\ran{1}{21}$}:
		Here, the extremes are $21$ and $1$, which by Lemma~\ref{no-edge-lemma} remain non-adjacent. SCIP proves their incompatibility, further splitting the range into two final branches:
		\begin{itemize}
			\item $\ran{2}{21}$: This range contains exactly $20 = n$ values with the sum $\sum_{k=2}^{21}k = 230 \equiv 2 \pmod{3}$, which violates the Triangle-Shake Corollary~\ref{triangle-shake-lemma}, immediately discarding this branch;
			\item $\ran{1}{20}$: This range also contains exactly $20$ values, but satisfies the divisibility condition $\left( \sum_{k=1}^{20}k = 210 \equiv 0 \pmod{3} \right)$, so we must verify this configuration directly. Since the maximum admissible $K_3$-degree in this branch is $20$, Lemma~\ref{no-edge-lemma} guarantees that $v_1$ and $v_{20}$ are non-adjacent, while Corollary~\ref{intersection-corollary} yields $c = |N(v_1) \cap N(v_{20})| \le 4$, which constrains $|N(v_1) \cap B| \ge 4$. SCIP confirms infeasibility of this relaxed model, closing the branch.
		\end{itemize}
		
		\item \emph{Subcase $\ran{0}{20}$}: Similarly, we fix the extremes $20$ and $0$. The solver proves their incompatibility, producing two branches:
		\begin{itemize}
			\item $\ran{1}{20}$: This branch is identical to the one resolved above;
			\item $\ran{0}{19}$: This range contains exactly $20$ values, although the sum $\sum_{k=0}^{19}k = 190 \equiv 1 \pmod{3}$, again violating Corollary~\ref{triangle-shake-lemma}. 
		\end{itemize}
	\end{itemize}

    This fully exhausts all possibilities for $n=20$. The structure of the subcase eliminations is illustrated as a directed acyclic graph (DAG) in Figure~\ref{fig:tree_n20}.
	
	\begin{figure}[H]
		\centering
		\begin{tikzpicture}[
			x=3.5cm, y=-2cm,
			every node/.style={rectangle, rounded corners, draw, align=center, fill=backcolour, minimum width=2.5cm, minimum height=0.8cm, font=\small},
			elab/.style={draw=none, fill=none, font=\footnotesize, text=black, minimum width=0, minimum height=0}
			]
			
			\node (n0_22) at (0, 0) {$\ran{0}{22}$};
			\node (n0_21) at (0, 1) {$\ran{0}{21}$};
			
			\node (n1_21) at (-0.8, 2) {$\ran{1}{21}$};
			\node (n0_20) at (0.8, 2) {$\ran{0}{20}$};
			
			\node (n2_21) at (-1.6, 3) [fill=gray!20] {$\ran{2}{21}$ \\[-0.5ex] \footnotesize $\Sigma \equiv 2 \pmod 3$};
			\node (n1_20) at (0, 3) {$\ran{1}{20}$ \\[-0.5ex] \footnotesize SCIP};
			\node (n0_19) at (1.6, 3) [fill=gray!20] {$\ran{0}{19}$ \\[-0.5ex] \footnotesize $\Sigma \equiv 1 \pmod 3$};
			
			\draw[-latex, thick] (n0_22) -- node[elab, right=0pt] {cut $22$} (n0_21);
			
			\draw[-latex, thick] (n0_21) -- node[elab, above left=-2pt] {cut $0$} (n1_21);
			\draw[-latex, thick] (n0_21) -- node[elab, above right=-2pt] {cut $21$} (n0_20);
			
			\draw[-latex, thick] (n1_21) -- node[elab, above left=-2pt] {cut $1$} (n2_21);
			\draw[-latex, thick] (n1_21) -- node[elab, above right=-2pt] {cut $21$} (n1_20);
			
			\draw[-latex, thick] (n0_20) -- node[elab, above left=-2pt] {cut $0$} (n1_20);
			\draw[-latex, thick] (n0_20) -- node[elab, above right=-2pt] {cut $20$} (n0_19);
			
		\end{tikzpicture}
		\caption{Subcase elimination DAG for $n=20$. Branches represent eliminating an extreme value (top or bottom) after showing its incompatibility. Terminal nodes (gray) are either resolved analytically ($\Sigma \not\equiv 0 \pmod 3$) or computationally via SCIP.}
		\label{fig:tree_n20}
	\end{figure}
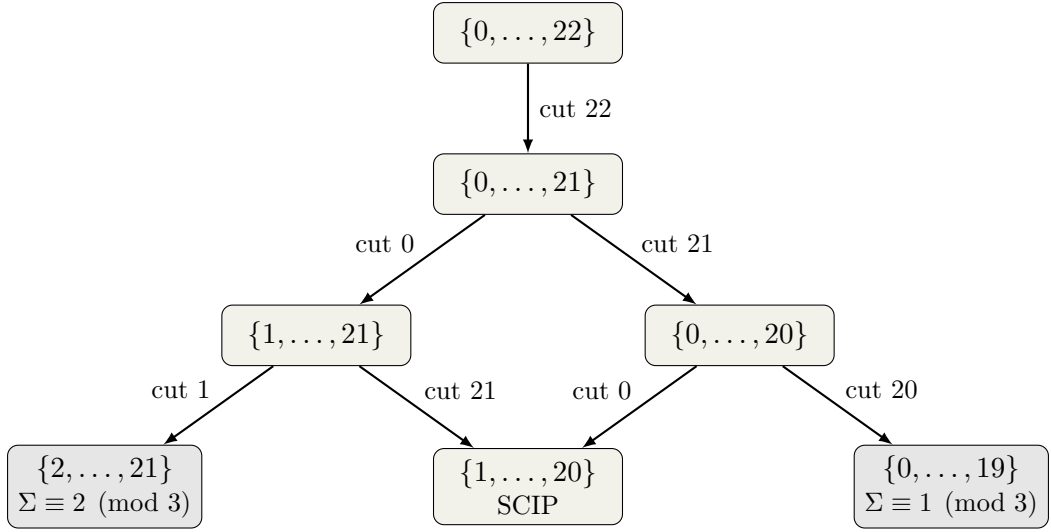
    
    \begin{table}[htbp]
        \centering
        \begin{tabular}{@{}lccrrr@{}}
        \toprule
            Evaluated range & $|N(u) \cap N(v)|$ & Max Depth & Nodes & LP iter & Time \\
        \midrule
            \multicolumn{6}{@{}l}{\textbf{Case $\boldsymbol{n=20}$}} \\
            $\{0, \ldots, 21, \mathbf{\underline{22}}\}$ & - & 66 & 9.7M & 3502M & 93\,h \\
            $\{\mathbf{\underline{0}}, 1, \ldots, 20, \mathbf{\underline{21}}\}$ & $\le 3$ & 0 & 1 & 1.4k & 0.12\,s \\
            $\{\mathbf{\underline{1}}, 2, \ldots, 20, \mathbf{\underline{21}}\}$ & $\le 4$ & 29 & 10.7k & 5.3M & 7\,m \\
            $\{\mathbf{\underline{0}}, 1, \ldots, 19, \mathbf{\underline{20}}\}$ & $\le 4$ & 30 & 39.0k & 13.1M & 13.5\,m \\
            $\{\mathbf{\underline{1}}, 2, \ldots, 19, \mathbf{\underline{20}}\}$ & $\le 4$ & 40 & 999k & 349M & 7.9\,h \\
        \bottomrule
        \end{tabular}
        \caption{Computational results for $n = 20$.}\label{table-20}
    \end{table}
	
	\medskip
	\noindent\textbf{Case $\boldsymbol{n=21}$.}
    The computational results for this case are provided in Table~\ref{table-21} and the subcase elimination DAG is illustrated in Figure~\ref{fig:tree_n21}.
	Starting from the full range $\ran{0}{22}$, we immediately employ the two-vertex fixing strategy to show that $0$ and $22$ are incompatible.
	By Lemma~\ref{no-edge-lemma}, they are non-adjacent. By Corollary~\ref{intersection-corollary}, their common neighborhood size is bounded by $c = |N(v_0) \cap N(v_{22})| \le 3$. We anchor at $v_{22}$, place $v_0 \in B$, and launch SCIP using the \emph{relaxed} formulation. The solver quickly confirms infeasibility, splitting the search into two main branches: $\ran{1}{22}$ and $\ran{0}{21}$.

    For both of these new ranges, the difference between the extremes is $21$. By Corollary~\ref{intersection-corollary}, we again obtain the upper bound $c \le 3$. We show that their extremes are incompatible using the \emph{relaxed} formulation, which branches the search further.

    When we reach the ranges $\ran{2}{22}$, $\ran{1}{21}$, and $\ran{0}{20}$, the difference between the extremes is $20$. Here, the relaxed formulation becomes intractable. Therefore, we switch to the \emph{exact} formulation. For each pair of extremes, we use both upper and lower bounds on $c$. By Corollary~\ref{intersection-corollary}, $c \le 4$. By Corollary~\ref{intersection-lower-bound-corollary}, since $n=21$ and $|d_1-d_2| \ge 20$, we have $c \ge 2$. Thus, we explicitly split each problem into three \emph{exact} subcases by fixing $c \in \{2, 3, 4\}$. We launch SCIP separately for each configuration and confirm infeasibility. 
    
	This process is performed recursively until we reach ranges containing fewer than $n=21$ admissible values, triggering the pigeonhole principle stopping condition. The full structural breakdown is as follows:
	
	\begin{itemize}
		\item \emph{Subcase $\ran{1}{22}$}: We show that $1$ and $22$ are incompatible (using relaxed $c \le 3$), splitting into two cases. Note that both resulting ranges have exactly $n=21$ values, nevertheless, the divisibility condition is satisfied, so we have to proceed with SCIP further:
		\begin{itemize}
			\item $\ran{2}{22}$: we show that $2$ and $22$ are incompatible using exact $c \in \{2, 3, 4\}$;
			\item $\ran{1}{21}$: we show that $1$ and $21$ are incompatible using exact $c \in \{2, 3, 4\}$.
		\end{itemize}
		
		\item \emph{Subcase $\ran{0}{21}$}: We show that $0$ and $21$ are incompatible (using relaxed $c \le 3$), splitting into:
		\begin{itemize}
			\item $\ran{1}{21}$: already resolved just above;
			\item $\ran{0}{20}$: we show that $0$ and $20$ are incompatible using exact $c \in \{2, 3, 4\}$.
		\end{itemize}
	\end{itemize}
	
	\begin{figure}[H]
		\centering
		\adjustbox{max width=\textwidth}{
			\begin{tikzpicture}[
				x=2.5cm, y=-2cm,
				every node/.style={rectangle, rounded corners, draw, align=center, fill=backcolour, minimum width=2cm, minimum height=0.7cm, font=\footnotesize},
				elab/.style={draw=none, fill=none, font=\footnotesize, text=black, minimum width=0, minimum height=0}
				]
				
				\node (n0_22) at (0, 0) {$\ran{0}{22}$};
				
				\node (n1_22) at (-0.9, 1) {$\ran{1}{22}$};
				\node (n0_21) at (0.9, 1) {$\ran{0}{21}$};
				
				\node (n2_22) at (-1.8, 2) {$\ran{2}{22}$};
				\node (n1_21) at (0, 2) {$\ran{1}{21}$};
				\node (n0_20) at (1.8, 2) {$\ran{0}{20}$};
				
				\node (n3_22) at (-2.7, 3) [fill=gray!20] {$\ran{3}{22}$ \\[-0.5ex] \tiny $<n$};
				\node (n2_21) at (-0.9, 3) [fill=gray!20] {$\ran{2}{21}$ \\[-0.5ex] \tiny $<n$};
				\node (n1_20) at (0.9, 3) [fill=gray!20] {$\ran{1}{20}$ \\[-0.5ex] \tiny $<n$};
				\node (n0_19) at (2.7, 3) [fill=gray!20] {$\ran{0}{19}$ \\[-0.5ex] \tiny $<n$};
				
				\draw[-latex, thick] (n0_22) -- node[elab, above left=-2pt] {cut $0$} (n1_22);
				\draw[-latex, thick] (n0_22) -- node[elab, above right=-2pt] {cut $22$} (n0_21);
				
				\draw[-latex, thick] (n1_22) -- node[elab, above left=-2pt] {cut $1$} (n2_22);
				\draw[-latex, thick] (n1_22) -- node[elab, above right=-2pt] {cut $22$} (n1_21);
				
				\draw[-latex, thick] (n0_21) -- node[elab, above left=-2pt] {cut $0$} (n1_21);
				\draw[-latex, thick] (n0_21) -- node[elab, above right=-2pt] {cut $21$} (n0_20);
				
				\draw[-latex, thick] (n2_22) -- node[elab, above left=-2pt] {cut $2$} (n3_22);
				\draw[-latex, thick] (n2_22) -- node[elab, above right=-2pt] {cut $22$} (n2_21);
				
				\draw[-latex, thick] (n1_21) -- node[elab, above left=-2pt] {cut $1$} (n2_21);
				\draw[-latex, thick] (n1_21) -- node[elab, above right=-2pt] {cut $21$} (n1_20);
				
				\draw[-latex, thick] (n0_20) -- node[elab, above left=-2pt] {cut $0$} (n1_20);
				\draw[-latex, thick] (n0_20) -- node[elab, above right=-2pt] {cut $20$} (n0_19);
				
			\end{tikzpicture}
		}
		\caption{Subcase elimination DAG for $n=21$. Branches split by eliminating an extreme value after confirming incompatibility with SCIP. Terminal nodes all contain fewer than $n=21$ values and are closed by the pigeonhole principle.}
		\label{fig:tree_n21}
	\end{figure}
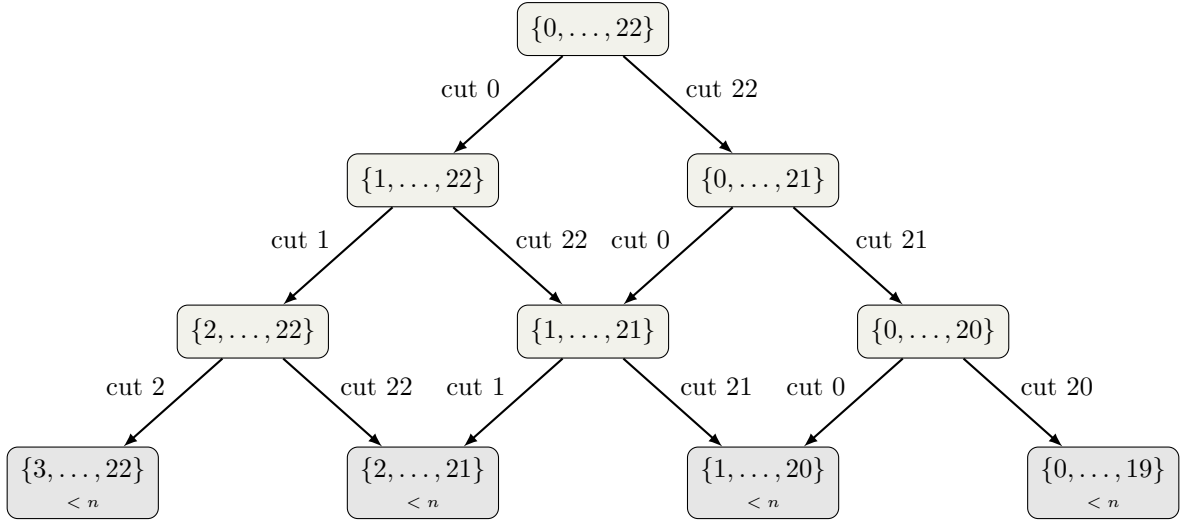

    \begin{table}[htbp]
        \centering
        \begin{tabular}{@{}lccrrr@{}}
        \toprule
            Evaluated range & $|N(u) \cap N(v)|$ & Max Depth & Nodes & LP iter & 
            Time \\
        \midrule
            \multicolumn{6}{@{}l}{\textbf{Case $\boldsymbol{n=21}$}} \\
            $\{\mathbf{\underline{0}}, 1, \ldots, 21, \mathbf{\underline{22}}\}$ & $\le 3$ & 22 & 967 & 454k & 38.43\,s \\
            $\{\mathbf{\underline{0}}, 1, \ldots, 20, \mathbf{\underline{21}}\}$ & $\le 3$ & 49 & 304k & 83.9M & 103\,m \\
            $\{\mathbf{\underline{1}}, 2, \ldots, 21, \mathbf{\underline{22}}\}$ & $\le 3$ & 53 & 157k & 64.6M & 100\,m \\
        \midrule
            $\{\mathbf{\underline{2}}, 3, \ldots, 21, \mathbf{\underline{22}}\}$ & $= 4$ & 0 & 1 & 1.3k & 0.17\,s \\
            $\{\mathbf{\underline{2}}, 3, \ldots, 21, \mathbf{\underline{22}}\}$ & $= 3$ & 40 & 37.3k & 15.1M & 21\,m \\
            $\{\mathbf{\underline{2}}, 3, \ldots, 21, \mathbf{\underline{22}}\}$ & $= 2$ & 23 & 3.5k & 1.5M & 124.22\,s \\
            \midrule
            $\{\mathbf{\underline{1}}, 2, \ldots, 20, \mathbf{\underline{21}}\}$ & $= 4$ & 0 & 29 & 12.9k & 8.21\,s \\
            $\{\mathbf{\underline{1}}, 2, \ldots, 20, \mathbf{\underline{21}}\}$ & $= 3$ & 53 & 813k & 313M & 7.6\,h \\
            $\{\mathbf{\underline{1}}, 2, \ldots, 20, \mathbf{\underline{21}}\}$ & $= 2$ & 22 & 2.3k & 508k & 46.22\,s \\
            \midrule
            $\{\mathbf{\underline{0}}, 1, \ldots, 19, \mathbf{\underline{20}}\}$ & $= 4$ & 0 & 36 & 16.8k & 5.63\,s \\
            $\{\mathbf{\underline{0}}, 1, \ldots, 19, \mathbf{\underline{20}}\}$ & $= 3$ & 41 & 287k & 95.9M & 104\,m \\
            $\{\mathbf{\underline{0}}, 1, \ldots, 19, \mathbf{\underline{20}}\}$ & $= 2$ & 19 & 1.2k & 326k & 22.06\,s \\
        \bottomrule
        \end{tabular}
        \caption{Computational results for $n = 21$.}\label{table-21}
    \end{table}
	
	\medskip
	\noindent\textbf{Case $\boldsymbol{n=22}$.} The computational results for this case are provided in Table~\ref{table-22} and the branching strategy is illustrated in Figure~\ref{fig:tree_n22}.
	Starting from the full range $\ran{0}{22}$, we employ the two-vertex fixing strategy to show that $0$ and $22$ are incompatible. By Lemma~\ref{no-edge-lemma}, they are non-adjacent. Because this case is computationally very demanding, we immediately utilize the \emph{exact} formulation. By Corollary~\ref{intersection-corollary}, $c \le 3$, and by Corollary~\ref{intersection-lower-bound-corollary}, since $n=22$ and $|d_1-d_2|=22 \ge 21$, we have $c \ge 2$. Thus, we explicitly split the problem into two exact subcases by fixing $c \in \{2, 3\}$. SCIP confirms infeasibility for both configurations, splitting the main search into two branches: $\ran{1}{22}$ and $\ran{0}{21}$.

    The branch $\ran{1}{22}$ contains exactly $22$ values, with sum $\sum_{k=1}^{22}k = 253 \equiv 1 \pmod{3}$. This violates the Triangle-Shake Lemma (Corollary~\ref{triangle-shake-lemma}) and is thus closed analytically.
    
	We proceed only with the branch $\ran{0}{21}$. This is the final hurdle, containing exactly $n=22$ values and satisfying divisibility. We show that $0$ and $21$ are incompatible. Again, we use the exact formulation. The difference is $|21-0| = 21$, so Corollaries~\ref{intersection-corollary} and~\ref{intersection-lower-bound-corollary} tightly bound the common neighborhood to $c \in \{2, 3\}$. We anchor at $v_{21}$ and launch the final SCIP runs for these two exact configurations. Both reach infeasibility.

    This was computationally the hardest case evaluated under the two-vertex fixing strategy, with the exact subcase $c=3$ for $\ran{0}{21}$ requiring the longest individual SCIP run among all two-vertex models (66 hours, see Table~\ref{table-22}). We remark that our initial evaluation of $n=20$ using only a single fixed anchor required over 93 hours of computation just to eliminate the first boundary value (see Table~\ref{table-20}). Encountering such severe computational bottlenecks was precisely what motivated the development of the advanced two-vertex strategy and the exact subcase splitting employed throughout these larger cases.
    
	\begin{figure}[H]
		\centering
		\begin{tikzpicture}[
			x=2.5cm, y=-2cm,
			every node/.style={rectangle, rounded corners, draw, align=center, fill=backcolour, minimum width=2.5cm, minimum height=0.8cm, font=\small},
			elab/.style={draw=none, fill=none, font=\footnotesize, text=black, minimum width=0, minimum height=0}
			]
			
			\node (n0_22) at (0, 0) {$\ran{0}{22}$};
			
			\node (n1_22) at (-1.0, 1) [fill=gray!20] {$\ran{1}{22}$ \\[-0.5ex] \footnotesize $\Sigma \equiv 1 \pmod 3$};
			\node (n0_21) at (1.0, 1) {$\ran{0}{21}$};
			
			\node (n1_21) at (0, 2) [fill=gray!20] {$\ran{1}{21}$ \\[-0.5ex] \footnotesize $<n$};
			\node (n0_20) at (2.0, 2) [fill=gray!20] {$\ran{0}{20}$ \\[-0.5ex] \footnotesize $<n$};
			
			\draw[-latex, thick] (n0_22) -- node[elab, above left=-2pt] {cut $0$} (n1_22);
			\draw[-latex, thick] (n0_22) -- node[elab, above right=-2pt] {cut $22$} (n0_21);
			
			\draw[-latex, thick] (n0_21) -- node[elab, above left=-2pt] {cut $0$} (n1_21);
			\draw[-latex, thick] (n0_21) -- node[elab, above right=-2pt] {cut $21$} (n0_20);
			
		\end{tikzpicture}
		\caption{Subcase elimination tree for $n=22$. After evaluating the primary case, only one valid branch requires further computation.}
		\label{fig:tree_n22}
	\end{figure}
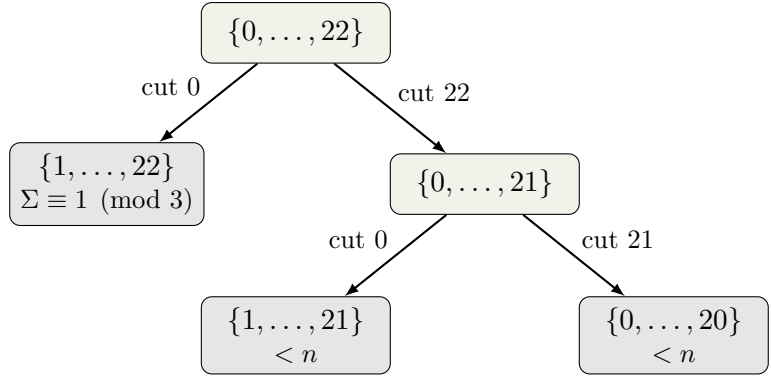

    \begin{table}[htbp]
        \centering
        \begin{tabular}{@{}lccrrr@{}}
        \toprule
            Evaluated range & $|N(u) \cap N(v)|$ & Max Depth & Nodes & LP iter & Time \\
        \midrule
            \multicolumn{6}{@{}l}{\textbf{Case $\boldsymbol{n=22}$}} \\
            $\{\mathbf{\underline{0}}, 1, \ldots, 21, \mathbf{\underline{22}}\}$ & $= 3$ & 50 & 15.1k & 4.8M & 6\,m \\
            $\{\mathbf{\underline{0}}, 1, \ldots, 21, \mathbf{\underline{22}}\}$ & $= 2$ & 45 & 73.0k & 26.7M & 37\,m \\
            \midrule
            $\{\mathbf{\underline{0}}, 1, \ldots, 20, \mathbf{\underline{21}}\}$ & $= 3$ & 66 & 8.25M & 2861M & 66\,h \\
            $\{\mathbf{\underline{0}}, 1, \ldots, 20, \mathbf{\underline{21}}\}$ & $= 2$ & 62 & 2.32M & 825M & 18\,h \\
        \bottomrule
        \end{tabular}
        \caption{Computational results for $n = 22$.}\label{table-22}
    \end{table}
	
	\section{Concluding remarks and open problems}\label{sec:conclusion}
	
	The exhaustive computational search presented in this paper establishes that there are no $8$-regular $K_3$-irregular graphs, thereby proving Theorem~\ref{thm-main}. Together with previous findings, this completely resolves the existence problem for all $r \le 8$, pushing the lower bound on the regularity of such graphs to $r = 9$. Thus, we have established Theorem~\ref{thm-threshold}, confirming that an $r$-regular $K_3$-irregular graph exists if and only if $r \ge 9$.
	
	In a recent study~\cite{Hak:25}, the first known example of a $9$-regular $K_3$-irregular graph was found; it has $n=24$ vertices. With the non-existence result for $r \leq 8$ confirmed, a natural open problem is to determine the absolute minimum order $n$ for which a regular $K_3$-irregular graph exists.

    Remarkably, our non-existence result for $r \le 8$ immediately implies a strict lower bound on the order of \emph{any} regular $K_3$-irregular graph. It is a known fact that the complement of an $r$-regular $K_3$-irregular graph on $n$ vertices is an $(n - 1 - r)$-regular $K_3$-irregular graph~\cite{Hak:25,Nair:94}. Thus, we can always assume without loss of generality that $r \le \lfloor (n-1)/2 \rfloor$. For $n \le 19$, this implies $r \le 9$. However, $r$ cannot be $9$ for odd $n \le 19$ by the handshaking lemma, and if $n \le 18$ is even and $r=9$, its complement has regularity $n - 1 - 9 \le 8$, which we have proven cannot exist. 

    \begin{theorem}\label{thm-n-bound}
        A regular $K_3$-irregular graph on $n$ vertices exists only if $n \ge 20$.
    \end{theorem}

    This leaves us with a fascinating dichotomy regarding the absolute smallest regular $K_3$-irregular graph. Stevanovi\'c et al.~\cite{reg-triangle:24} constructed an example for $r=10$ on $n=21$ vertices. If a $9$-regular $K_3$-irregular graph on $n=20$ vertices exists, then it (and its $10$-regular complement) would be the absolute smallest regular $K_3$-irregular graph in terms of both order ($n=20$) and regularity ($r=9$).

    If, however, no such graph exists for $n=20$, then the minimum order and minimum regularity are achieved by different graphs: the minimum order would be exactly $n=21$ (realized by the $10$-regular graph from~\cite{reg-triangle:24}), while the minimum regularity would remain $r=9$, realized by a graph on $n=22$ vertices (if it exists) or our known example on $n=24$ vertices~\cite{Hak:25}.
	
    Consequently, resolving the existence of $9$-regular $K_3$-irregular graphs on $n=20$ and $n=22$ vertices is the final step to completely determining the smallest regular $K_3$-irregular graphs. We note that the search space for these cases is vastly larger than the cases resolved in this paper, requiring significant computational resources or further theoretical breakthroughs. We leave this as an open computational challenge for future research.
    
	\vspace{0.5cm}
	
	{\bf Acknowledgements}
	
	\vspace{0.5cm}
	
	The authors thank the Ukrainian Armed Forces for keeping Leliukhivka and Kyiv safe, which gave us the opportunity to work on this paper. We are also grateful to Mykola Korablov for suggesting the use of integer linear programming to tackle this problem.

    \vspace{0.5cm}

    \noindent\textbf{Declaration of Generative AI and AI-assisted technologies in the writing process} 
    
    During the preparation of this work, the authors used LLM-based assistants (Gemini 3.1 Pro and Claude 4.6 Opus) to improve language readability, polish mathematical prose, and assist with code refactoring. After using these tools, the authors reviewed and edited the content as needed and take full responsibility for the content of the publication.
    
	

    

\begin{thebibliography}{10}
		\footnotesize\itemsep=0pt
		\providecommand{\url}[1]{#1}
		\providecommand{\urlprefix}{}
		\providecommand{\eprint}[2][]{\href{https://arxiv.org/abs/#2}{arXiv:#2}}
		
		\bibitem{dist-triangle:2024}
		Berikkyzy Z., Bjorkman B., Blake H.S., Jahanbekam S., Keough L., Moss K., Rorabaugh D., Shan S., Triangle-degree and triangle-distinct graphs, \href{https://doi.org/10.1016/j.disc.2023.113695}{\textit{Discrete Math.}} \textbf{347} (2024), Paper No. 113695, 8.

		\bibitem{Caro2024}
		Caro~Y., Lauri~J., Mifsud~X., Yuster~R., Zarb~C., Flip colouring of graphs, \href{https://doi.org/10.1007/s00373-024-02838-w}{\textit{Graphs Combin.}} \textbf{40} (2024), article no.~110, 24~pages.

        \bibitem{CaroMifsud2025}
		Caro~Y., Mifsud~X., On $(r,c)$-constant, planar and circulant graphs, \href{https://doi.org/10.7151/dmgt.2551}{\textit{Discuss. Math. Graph Theory}} \textbf{45} (2025), 707--723.
        
		\bibitem{Chartrand2016}
		Chartrand~G., Highly irregular, in Graph theory---favorite conjectures and open problems.~1, \textit{Probl. Books in Math.}, \href{https://doi.org/10.1007/978-3-319-31940-7_1}{Springer}, Cham, 2016, 1--16.
		
		\bibitem{Char-Erd-Oell:88}
		Chartrand~G., Erd\H{o}s~P., Oellermann~O.R., How to define an irregular graph, \href{https://doi.org/10.2307/2686701}{\textit{College Math. J.}} \textbf{19} (1988), 36--42.
		
		\bibitem{Char:87}
		Chartrand~G., Holbert~K.S., Oellermann~O.R., Swart~H.C., {$F$}-degrees in graphs, \textit{Ars Combin.} \textbf{24} (1987), 133--148.
		
		\bibitem{ZenodoK3Irregular}
		Hak~A., Kozerenko~S., Serdiuk~A., EvolutionaryGraphs (2025), \newblock \href {https://doi.org/10.5281/zenodo.21877760} {\path{doi:10.5281/zenodo.21877760}}.
		
		\bibitem{IrregularILP}
		Hak~A., Regular K3-Irregular Graph ILP Generator (2026), \newblock \href{https://doi.org/10.5281/zenodo.23402146} {\path{doi:10.5281/zenodo.23402146}}.
	
		\bibitem{Hak:25}
		Hak~A., Kozerenko~S., Serdiuk~A., Regular $K_3$-irregular graphs, 2025, \eprint{2507.18776}.

        \bibitem{HakReg:25}
		Hak~A., Kozerenko~S., Lohvynov~D., Yarosh~Y., Regular $K_3$-regular graphs, accepted for publication in \textit{Discuss. Math. Graph Theory}, 2026, \eprint{2602.23517}.

        \bibitem{SCIP}
		Hojny~C. et al. \textit{The SCIP Optimization Suite 10.0.} Technical Report. Optimization Online, Nov. 2025. URL: \url{https://optimization-online.org/2025/11/the-scip-optimization-suite-10-0/}.
		
		\bibitem{Jajcay2025}
		Jajcay~R., Jooken~J., Porups\'{a}nszki~I., On vertex-girth-regular graphs: (non-)existence, bounds and enumeration, \href{https://doi.org/10.37236/13333}{\textit{Electron. J. Combin.}} \textbf{32} (2025), paper no.~P4.51, 27~pages.
        
		\bibitem{Nair:94}
		Nair~B.R., Vijayakumar~A., About triangles in a graph and its complement, \href{https://doi.org/10.1016/0012-365X(94)90385-9}{\textit{Discrete Math.}} \textbf{131} (1994), 205--210.
		
		\bibitem{Nair:96}
		Nair~B.R., Vijayakumar~A., Strongly edge triangle regular graphs and a conjecture of {K}otzig, \href{https://doi.org/10.1016/0012-365X(95)00077-A}{\textit{Discrete Math.}} \textbf{158} (1996), 201--209.
		
		\bibitem{reg-triangle:24}
		Stevanovi\'c~D., Ghebleh~M., Caporossi~G., Vijayakumar~A., Stevanovi\'c~S., On regular triangle-distinct graphs, \href{https://doi.org/10.1007/s40314-024-02854-9}{\textit{Comput. Appl. Math.}} \textbf{43} (2024), article no.~336, 19~pages.

        \bibitem{Zhang:26}
		Zhang~Z., Regular $K_3$-irregular graphs of every regularity at least nine, preprint, 2026, \eprint{2610.08478}.
	\end{thebibliography}
\end{document}